\documentclass[reqno]{amsart}
\usepackage{amsfonts,amsthm,amsmath,amssymb,amscd,amsmath,epsf,latexsym,enumerate,mathrsfs}
\usepackage{graphicx}
\usepackage[utf8x]{inputenc}
\newtheorem{theorem}{Theorem}
\newtheorem{proposition}[theorem]{Proposition}
\newtheorem{lemma}[theorem]{Lemma}
\usepackage{cleveref}
\newtheorem{definition}{Definition}

\newtheorem{example}{Example}
\newtheorem{remark}{Remark}

\newtheorem{problem}{Problem}

\newcommand{\bC}{\mathbb{C}}
\newcommand{\ee}{\end{equation}}

\begin{document}
          \numberwithin{equation}{section}

          \title[ Real zeros for polynomials obeying a five-term recurrence]
          {Special 5-term recurrence relations, banded Toeplitz matrices, and reality of zeros
}    
         \author[I.~Ndikubwayo]{Innocent Ndikubwayo}
\address{Department of Mathematics, Stockholm University, SE-106 91
Stockholm,         Sweden}
\email {innocent@math.su.se, ndikubwayo@cns.mak.ac.ug}



\keywords{ recurrence relation, Banded Toeplitz matrix, hyperbolic} 
\subjclass[2010]{Primary 12D10,\; Secondary   26C10, 30C15}

\maketitle

\begin{abstract} 
Below we establish the conditions guaranteeing the  reality of all the zeros of polynomials $P_n(z)$ in the  polynomial sequence  $\{P_n(z)\}_{n=1}^{\infty}$ satisfying a five-term recurrence relation
 $$P_{n}(z)= zP_{n-1}(z) + \alpha P_{n-2}(z)+\beta P_{n-3}(z)+\gamma P_{n-4}(z),$$ with
the standard initial conditions $$P_0(z) = 1, P_{-1}(z) = P_{-2}(z) =P_{-3}(z) = 0,$$ where $\alpha, \beta, \gamma$  are  real coefficients, $\gamma\neq 0$ and $z$ is a complex variable.
We interprete this sequence of polynomials as principal minors of an appropriate banded  Teoplitz matrix whose associated Laurent polynomial $b(z)$ is holomorphic in  $\mathbb{C}\setminus \{0\}$.  We show that when either the critical points of $b(z)$ are all real; or when they are two real and one pair of complex conjugate critical points with some extra conditions on the parameters, the set $b^{-1}(\mathbb{R})$ contains  a Jordan curve with $0$ in its interior and in some cases a non-simple curve enclosing $0$. The presence of the said  curves is necessary and sufficient for every polynomial in the sequence $\{P_n(z)\}_{n=1}^{\infty}$ to be hyperbolic (real-rooted).

\end{abstract}

\section{Introduction and Basic notation}

It is a natural problem to investigate location of zeros of polynomials defined by linear 
recursive relations. A classic  and extensively studied recursion is the three-term recursive formula since it provides a necessary condition for a sequence of polynomials to be orthogonal. 
Orthogonal polynomial sequences have many interesting properties and applications. In particular, every polynomial in the orthogonal polynomial sequence is real-rooted and its zeros lie on some interval of orthogonality,  for details see \cite{MG, AMS}.
However, not much is known about the location of zeros of polynomials generated by five-term recurrences. In this work, we study the zeros of polynomials $P_n(z)$ generated by a certain five-term recurrence and we are interested in situations where each of the polynomials in $\{P_n(z)\}_{n=1}^{\infty}$ has all  real roots. This particular relation is connected to banded Toeplitz matrices.

Toeplitz matrices and operators arise in many areas  of pure and applied mathematics for example, probability theory, time series analysis, harmonic analysis, signal theory and image processing among others \cite{DAB}.
A matrix $T$ is called Toeplitz (of finite, semi-infinite or bi-infinite order) if  $T$  has entries that are constant along its diagonals. By defining $T$ entrywise, we have $T_{ij}= a_{i-j},~~~~ \forall~~ i, j =0,1,\dots,$.  A semi-infinite matrix of the same form is known as a Toeplitz operator, and a bi-infinite matrix of this kind is called a Laurent operator \cite{LNTM}. Now from the way $T$ is defined, it becomes clear that $T$ is completely determined by its entries in the first row and the first column i.e., the set
\begin{eqnarray} \label{lovenc}
\{ a_k \}_{k={-\infty}}^{\infty} = \{ \dots, a_{-2}, a_{-1}, a_0, a_1, a_2, \dots\}.
\end{eqnarray}
In general, these entries $a_k$ are complex numbers.
A Toeplitz matrix $T$ is said to be\textit{ banded} if there is an integer \textbf{$d$} such that $a_\ell=0$ whenever $|\ell| > d.$ In this case, we say that $T$ is a banded Toeplitz matrix with bandwidth $d$. To put it differently, a Toeplitz matrix is banded if and only if at most finitely many of the numbers in \eqref{lovenc} are nonzero. 
\begin{definition}
The \textit{symbol} of a Toeplitz matrix or Toeplitz operator or Laurent 
operator is the function of the form
 \begin{eqnarray*}\label{stella11}
f_T(z)= \sum_{k} a_k z^k.
\end{eqnarray*}
This is a finite sum  or an infinite series depending on the context.
\end{definition}
\begin{example}  The matrix
\begin{equation}\label{ind1}
T=
\begin{pmatrix}
    0 & 2i & -1 & 2 & 0 & 0\\
    0 & 0 & 2i & -1 & 2 & 0 \\
    -4 & 0 & 0 & 2i & -1 & 2 \\
    -2i & -4 & 0 & 0 & 2i & -1\\
    0 & -2i & -4 & 0 & 0 & 2i \\
    0 & 0 & -2i & -4 & 0 & 0 \\
\end{pmatrix}
\end{equation}
is a $6 \times 6 $ banded  Toeplitz matrix with bandwidth $3$.\end{example}  
Associated  to $T$  in \eqref{ind1} is the symbol  
\begin{equation}\label{Ins}
f_T(z)=\frac{2}{z^3}-\frac{1}{z^2}+\frac{2i}{z}-4z^2-2iz^3.
\end{equation}
Since the matrix $T$ defined  in (\ref{ind1}) is banded, the symbol $f_T(z)$ is a finite linear combination of positive and negative powers of $z$, also referred to as  a \textit{Laurent polynomial.} We can think of $f_T(z)$ as a complex-valued function, thus in \eqref{Ins}, $f_T(z)$ is holomorphic in $\mathbb{C}\setminus \{0\}$. \\
Consider a Laurent polynomial  of the form
\begin{eqnarray}\label{stella11a}
b(z)= \sum_{k=-r}^{s} a_k z^k  ~~~~\mbox{with}~~~~a_{-r}a_s \neq 0 ~~~~\mbox{and}~~~~ r,s \in \mathbb{N}
\end{eqnarray}
where $a_k$ are complex numbers.
In what follows, we use the notation $T_n$ to denote an $n\times n$ finite section of the infinite dimensional Toeplitz matrix $T$ and write $T_n(b)$ for the matrix $T_n$ whose associated Laurent polynomial is $b(z).$   

The study of the asymptotic behaviour of the determinant of $T_n(b)$ as $n\to \infty$ goes back to the work of Szeg{\H o} \cite{GS1, GS2} but the more recent progress is contained  in the work of Schmidt and Spitzer \cite{KD}, where they characterize the limiting behaviour of the spectrum of $T_n(b)$ as follows:

Denote by $\operatorname{spec}(T_n(b))$, the spectrum of $T_n(b)$. Then 
$$\operatorname{spec}(T_n(b)):= \{ \lambda \in \mathbb{C}: \operatorname{det}(T_n(b)-\lambda I_n)=0 \},$$
where $I_n$ is an  $n \times n$ identity matrix.
The eigenvalue counting measures of  $T_n(b)$ are given by
$$\mu_n:=\frac{1}{n}\sum_{k=1}^n\delta_{\lambda_{k,n}},$$
where $\lambda_{1,n}, \lambda_{2,n}, \dots, \lambda_{n,n}$ are the eigenvalues of $T_n(b)$ (repeated according to their algebraic multiplicity) and $\delta_a$ is the Dirac measure supported at $\{a\}$. The sequence $\{\mu_n\}$ converges weakly to a limiting measure $\mu$, as $n \to \infty$. By definition, the limiting measure $\mu$ is supported on the limiting set
$$\wedge(b):= \{ \lambda \in \mathbb{C}:\lim_{n \to \infty} \operatorname{dist}(\lambda, \operatorname{spec}(T_n(b)))=0 \},$$
where $\operatorname{dist}(x,y)$ is the distance between $x$ and $y$, see \cite{KD} for details.

\begin{definition}\cite{GM} 
A curve $\Upsilon \subset \mathbb{C}$ is a continuous mapping from a closed interval to a topological space. We say that $\Upsilon$  is an analytic curve/ arc if every point of $\Upsilon$ has an open neighbourhood  $U$ for which there is an onto conformal map $h: \mathbb{T} \to U$ (where $ \mathbb{T} \subset \mathbb{C}$ is the unit disc) such that $\mathbb{T}\cap \mathbb{R}$ is mapped by $h$ onto $\mathbb{U}\cap \Upsilon$. 
A point  $\lambda  \in \wedge(b)$ is called an exceptional point if  either $\lambda$ is a branch point or if there is no open neighborhood $V$ of $\lambda$ such that $\wedge(b) \cap V$ is an analytic arc starting and terminating on the boundary $\partial{V}$. \end{definition}

For a fixed $\lambda \in \bC$, define a polynomial $R(z)=z^r(b(z)-\lambda)$. Further, let
\begin{eqnarray*}\label{eq:Ga2}
z_1(\lambda)\ge z_2(\lambda)\ge \ldots \ge z_r(\lambda)\ge z_{r+1}(\lambda)\ldots \ge z_{r+s}(\lambda)>0
\end{eqnarray*}
 be the $(r+s)$-tuple of the absolute values of all  (not necessarily distinct) roots of $R(z)$ in a non-increasing order.  Schmidt and Spitzer \cite{KD} gave a description of the limiting set $\wedge(b)$ as 
\begin{eqnarray}\label{eq:Gax}
\wedge(b):= \{\lambda \in \bC\;  |\;  z_r(\lambda)=z_{r+1}(\lambda)\}.
\end{eqnarray}
If among the roots of $R(z)$ there is a chain of equalities, then the labelling of roots within this chain is arbitrary.

From \eqref{eq:Gax}, we can deduce a topological description of the set  $\wedge(b)$ as a disjoint union of a finite number of (open) analytic curves and a finite number of exceptional points. Moreover, $\wedge(b)$ has no isolated points and is both compact and connected set (refer to \cite{MDA,BAG} for more details).
Valuable  references containing many results on the spectral properties of band Toeplitz matrices are the books of Böttcher and Silbermann \cite{TR1, TR2, TR3}.  Widom \cite{HW} also gives such an impressive introduction to this subject area.

In 2017, a new approach in understanding the spectral properties of band Toeplitz matrices was successfully used by Shapiro and Stampach \cite{BSF}  in the study of a class of complex semi-infinite banded Toeplitz matrices, whose spectra of their principal submatrices accumulate onto a real interval as the size of their submatrices grows to infinity. They proved that, a banded Toeplitz matrix belongs to the said class if and only if its Laurent polynomial has real values on a Jordan curve located in $\mathbb{C}\setminus\{0\}$, (where by a Jordan curve we mean a non-self- intersecting closed curve). Moreover, with the presence of such a Jordan curve, the spectra of all of the corresponding principal submatrices have to be real. 
 
Let $b(z)$ be the Laurent polynomial defined in \eqref{stella11a} and  $b^{-1}(\mathbb{R})$ be the pre-image of the real axis under $b(z)$. We now  state a theorem which will be central to our work. The details of its proof can be found in the same paper. 

\begin{theorem}\cite{BSF} \label{symb2}
Let $b(z)$ be a Laurent polynomial defined in \eqref{stella11a}. The following statements
are equivalent:\\

(i) $\wedge(b) \subset \mathbb{R}$;\\

(ii) the set $b^{-1}(\mathbb{R})$ contains (an image of) a Jordan curve;\\

(iii) for all $n \in \mathbb{N}$, $\operatorname{spec}(T_n(b)) \subset \mathbb{R}$.
\end{theorem}

We observe that Theorem \ref{symb2}  points out a nice feature of banded Toeplitz matrices, $(i)\Rightarrow (iii)$ i.e., the asymptotic reality of the eigenvalues automatically implies that all the eigenvalues of all principal submatrices are real. This is rather an interesting result simply because by \cite{LNTM}, there is no simple characterization of the eigenvalues for general Toeplitz matrices. 

The above theorem motivates the following problem discussed in the present paper.
\begin{problem}\label{prob:main}
{\rm} Consider the recurrence relation of the form
 \begin{eqnarray}\label{christ22}
P_n(z)= zP_{n-1}(z)+\alpha P_{n-2}(z)+\beta P_{n-3}(z)+\gamma P_{n-4}(z),
\end{eqnarray} 
with the standard initial conditions,
 \begin{eqnarray}  \label{almst}
  P_{0}(z)=1, P_{-1}(z)=P_{-2}(z)=P_{-3}(z)=0,
  \end{eqnarray} 
  where $\alpha, \beta, \gamma \in \mathbb{R}, \gamma\neq 0$ and $z$ is a complex variable. Give conditions on $(\alpha,\beta, \gamma)$ such that all the zeros of $P_{n}(z)$ will be real for all positive integer $n$. 
\end{problem}

Putting  $\beta=\gamma=0$, in \eqref{christ22} subject to \eqref{almst}, it is straight forward to show that the resulting 3-term recurrence relation generates real-rooted polynomials precisely when $\alpha \leq 0$. Similarly, in the case when  $\gamma=0$, it is not difficult to show by use of discriminants that the resulting 4-term relation generates real-rooted polynomials if and only if $\alpha \leq 0$ and $27\beta^2 \leq -\alpha^3.$ Moreover, by setting $\gamma=0, \beta=-1$ and $\alpha \in \mathbb{R}$, we obtain the 3-Conjecture of Egecioglu, Redmond and Ryavec \cite{OTC}. In this case, it is proven in \cite{OTC} (for sufficient condition) and later in \cite{BG} (for necessary condition) that for all positive integer $n$, all the zeros of $P_n(z)$ are real if and only $\alpha \leq 3$. In fact in the latter paper, it is also shown that if the inequality is strict, i.e, $\alpha < 3$, then the zeros of $P_{n+1}(z)$ and $P_n(z)$ interlace. To the best of my knowledge, no similar study has been carried out for Problem \ref{prob:main} and it is what we seek to address in this paper.\\
 
 Problem $1$ will be settled by finding out when condition of Theorem 1(ii) holds, and  then using the implication  $(ii)\Rightarrow (iii)$ of the same theorem to obtain the required results.

\medskip
To begin with, we make the following observations: The banded Teoplitz operator $T$ corresponding to the recurrence \eqref{christ22} subject to \eqref{almst} is 
\begin{equation}\label{shamim}
 T =
  \begin{pmatrix}
0 & -1 &  &  &  &  &  & \\
\alpha & 0 & -1 &  & & & & \\
-\beta & \alpha & 0 & -1 &  & & & \\
 \gamma& -\beta & \alpha & 0 & -1 &  & & \\
 &\gamma & -\beta & \alpha & 0 & -1 &  & \\
 & & \ddots& \ddots & \ddots & \ddots &\ddots  & \\
 & & & &  &  &  & \\
 &  &   &  &  &  &  &\ddots\\
\end{pmatrix}
\end{equation}
where the entries in the diagonals that are not shown are zeros. 

The Laurent polynomial associated to $T$ in \eqref{shamim} is 
\begin{equation}\label{shamim22}
b(z)= - \frac{1}{z}+ \alpha z  - \beta z^2 + \gamma z^3.
\end{equation}
It is clear that $b(z)$ is  holomorphic in $\mathbb{C}\setminus \{0\}$. 

\medskip

Let $T_n(b)$ be an $n \times n$ top-left  submatrix of $T$ in \eqref{shamim} associated with the Laurent polynomial $b(z)$ in \eqref{shamim22} and $Y_n(z)$ be its characteristic polynomial. Then, for each $n \in \mathbb{N}$, $Y_n(z)=P_n(z)$ where $P_n(z)$ is a polynomial generated by  the recurrence  \eqref{christ22} subject to \eqref{almst}. Hence from now onwards, we shall write $P_n(z)$ for $Y_n(z)$. Consequently, the zeros of $P_n(z)$ are precisely the eigenvalues of $T_n (b)$ i.e., the zeros of $\operatorname{det}(zI-T_n(b)).$
\medskip

From 
\begin{eqnarray*}\label{Mar1}
b(z)= - \frac{1}{z}+ \alpha z  - \beta z^2+ \gamma z^3,
\end{eqnarray*}
we have 
\begin{eqnarray*}
b'(z)=\frac{3\gamma z^4 -2\beta z^3 +  \alpha z^2+1}{z^2}= \frac{G(z)}{z^2}
\end{eqnarray*}
where 
\begin{eqnarray}\label{Flav1x}
G(z):=3\gamma z^4 -2\beta z^3 +  \alpha z^2+1.
\end{eqnarray}

 Interestingly, it turns out that the zeros of $G(z)$ play an important role in solving Problem $1$. The necessary and sufficient conditions under which all the polynomials $P_n(z)$ are hyperbolic can be stated in terms of the discriminants of a quartic, cubic and quadratic polynomials where the quartic polynomial in question is given by \eqref{Flav1x}. Let $\bigtriangleup(G)$  denote the discriminant of $G(z)$.

We now state the main result of the this paper. 

 \begin{theorem}\label{Main}
{\rm} Consider the five-term linear recurrence relation of the form 
\begin{equation*}\label{christ21}
P_n(z)= zP_{n-1}(z)+\alpha P_{n-2}(z)+\beta P_{n-3}(z)+\gamma P_{n-4}(z),
 \end{equation*} 
  with the standard initial conditions, 
\begin{equation*}
 P_{0}(z)=1, P_{-1}(z)=P_{-2}(z)=P_{-3}(z)=0,
\end{equation*} 
where $\alpha, \beta, \gamma$  are real coefficients, $\gamma \neq 0$  and $z$ is a complex variable. 
Then, for all positive integers $n$, all the zeros of $P_{n}(z)$ are real if and only if one of the  following two conditions is satisfied:\\
Either

{\rm (a)} The polynomial $G(z)=3 \gamma z^4-2\beta z^3+ \alpha z^2 +1$ is hyperbolic which is equivalent to
\begin{eqnarray*}
\alpha<0,~~~~~ -\frac{\alpha^2}{36} \leq \gamma \leq  \frac{\alpha^2}{12}~~\mbox{and}~ \bigtriangleup(G)\geq 0 \end{eqnarray*}
or

{\rm (b)} $G(z)$ has two distinct real zeros with different signs and a pair of complex conjugate zeros and, additionally 
$12\alpha\gamma-\beta^2 \geq 0$ and $64\gamma^3+(4\alpha\gamma-\beta^2)^2\geq 0$.\\

This condition in part $(b)$ is equivalent to
\begin{eqnarray*}\label{labb1}
 \gamma <0,~~~\bigtriangleup(G)<0,~~\bigtriangleup(W)\leq 0 ~~\mbox{and}~~\bigtriangleup(H)\geq 0, \end{eqnarray*}
where 
\begin{eqnarray*}\label{Auntkaduka}
W(z)=\alpha z^2+\beta z+ 3\gamma,
\end{eqnarray*}
and
\begin{eqnarray}\label{Auntkaduka}
H(z)=z^3+\lambda z+ \mu ~~~~~~\mbox{for}~~~~ \lambda= 2^{\frac{4}{3}}\gamma, ~~~ \mu=\frac{4\alpha \gamma -\beta^2}{3\sqrt{3}}.
\end{eqnarray}
\end{theorem}

In \textsection \ref{sect2} we discuss the  nets associated with real rational functions whose  critical points are all real;
in \textsection \ref{sect3} we discuss the standard discriminant of quartic polynomials and its application to solving Problem 1; in \textsection \ref{sect4} we prove Theorem $2(a)$; in \textsection \ref{sect5} we prove Theorem $2(b)$. Finally in \textsection \ref{sect6} we give some examples and numerical computations.
 
\section{Real rational functions}\label{sect2}
 Let $f_1, f_2 \in \mathbb{R}[z]$ and $f'_1, f'_2 $ be their respective derivatives. The Wronskian of $f_1, f_2$ is the polynomial $W(f_1, f_2):= f'_2f_1 - f_2 f'_1$. If $f_1 \not\equiv 0$, we consider a rational function $g:=\frac{f_2}{f_1}.$ This defines a map $g: \mathbb{C}P^1 \to \mathbb{C}P^1$ whose derivative is $\frac{W(f_1, f_2)}{f_1^2}.$
 
 For a non-constant rational function $g(z) = \frac{f_2(z)}{f_1(z)}$, where $f_1(z)$ and $f_2(z)$ are relatively prime polynomials, the degree of $g(z)$ is defined as the maximum of the degrees of $f_1(z)$ and  $f_2(z)$.
 
 \begin{definition}
 A  point $z_c$ is called a critical point of $g(z)$ if $g(z)$ fails to be injective in a neighbourhood of $z_c$. 
A critical value of $g(z)$ is the image of a critical point. The order of a critical point $z_c$ of $g(z)$ is the order of the zero of $g'(z)$ at $z_c$. 
A pole of  $g(z)$ is said to be a critical point of $g(z)$ of order $k$ if $z_c$ is a critical point of $1/g (z)$ of order $k$. A point $z = \infty$ is said to be a critical point of
$g(z)$ of order $k$ if $w = 0$ is a critical point of $ g(1/w)$ of order $k$.
\end{definition}

 Suppose that $f_1$ and $f_2$ have no common zeros. Then the critical points of $g$ in the  complex plane are the zeros of their Wronskian $W(f_1, f_2)$. Moreover, if  $d$ is the degree of $g(z)$, then by  Riemann–Hurwitz formula, $g$ has $2d-2$ critical points counting multiplicity in $\mathbb{C}P^1$, see \cite{TSS} for more details.

\subsection{Nets associated with real rational functions} \noindent

\medskip

Suppose that $g(z)$ is a real rational function, then $g:\mathbb{C}P^1 \to \mathbb{C}P^1$ has the property that it maps $\mathbb{R}P^1$ to $\mathbb{R}P^1$ where  $\mathbb{R}P^1 \subset \mathbb{C}P^1$ is $\mathbb{R}\cup \{\infty\}$, the extended real line. 
 Denote by  $\Gamma_g$ the set $\Gamma_g:=g^{-1}(\mathbb{R}P^1)$, the pre-image of $\mathbb{R}P^1$. Then, $\Gamma_g$ is called the \textit{net} of the rational function $g(z)$, see \cite{FS, MLN}. Moreover, $\Gamma_g$ partitions  $\mathbb{C}P^1$ into a finite number of regions. 
 
Since $g(z)$ is a real rational function, then $\Gamma_g$ is symmetric with respect to complex conjugation. If we additionally assume that all the critical points of $g(z)$ are real, then  $\Gamma_g$ consists of smooth curves which intersect only at the real critical points of $g(z)$. These intersection points are also called vertices of  $\Gamma_g$.
Observe that  $\mathbb{C}P^1 \setminus \Gamma_g$ is a union of $2d$ disjoint domains/regions, where $d$ is the degree of $g$. The closure of each region is homeomorphic to a disc, and the boundary of each region maps onto  $\mathbb{R}P^1$. This is because the regions (and their closures) are the pre-images of one of the two discs in $\mathbb{C}P^1\setminus \mathbb{R}P^1$ (or their closures) respectively, and there are no critical points in the interior of any  region, see \cite{FS} for more details. 

\section{The standard discriminant of quartic polynomials and its application to the problem }\label{sect3}

We begin this section by describing the properties of the zeros of real quartic polynomials  using the notion of discriminants. This is done in order to determine the possible behaviour of the zeros of $G(z)$ defined in \eqref{Flav1x}. It suffices to study the zeros of a family of  real quartic polynomials of the form 
\begin{eqnarray}\label{Flav22}
E(u)=u^4+ q u^2+r u +s.
\end{eqnarray}

\begin{definition}
Let $P(x)$ be a univariate polynomial of degree  $n$  with zeros $x_1, x_2,\ldots, x_n$ and leading coefficient $a_n$. The discriminant of $P(x)$  is  defined as  
$$\Delta(P)=a_n^{2n-2} \prod_{1\leq i<j\leq n}(x_i-x_j)^2.$$
\end{definition}
A straight forward computation gives the discriminant of $E(u)$  as
\begin{eqnarray*}
\Delta(E)=-4 q^3 r^2 - 27 r^4 + 16 q^4 s + 144 q r^2 s - 128 q^2 s^2 + 256 s^3.
\end{eqnarray*}

The following result holds. 
\begin{theorem} \cite{REL}  \label{lovence}
Let $E$ be as defined in \eqref{Flav22} and $\bigtriangleup(E)$ be its discriminant.
The following criteria  determine the character of zeros of $E$.
\begin{enumerate}[(a)]\item If $\bigtriangleup(E)<0$,  then $E$ has two distinct real zeros and a pair of non-real complex conjugate zeros.
\item If $\bigtriangleup(E)>0$, then all the zeros of $E$ are distinct. All are either real or complex conjugates as follows:
\begin{enumerate}[(i)]
\item $q <0, s>\frac{q^2}{4}$, the zeros are non-real (or complex-conjugates); \\
 $q <0, s<\frac{q^2}{4}$, the zeros are real;
\item $q \geq 0$, the zeros are non-real.
\end{enumerate}
\item If $\bigtriangleup(E)=0$, at least two zeros of $E$ are equal and
\begin{enumerate}[(i)]
\item $q <0, s>\frac{q^2}{4}$, two equal real zeros, two non-real;\\
 $q <0, -\frac{q^2}{12}<s<\frac{q^2}{4}$, all zeros are real, but two and only two are equal;\\
 $q <0, s=\frac{q^2}{4}$, then we have two pairs of equal real zeros;\\
$q <0, s=-\frac{q^2}{12}$, all zeros are real of which  three are equal;
\item $q >0, s>0, r \neq 0$, two equal real zeros, two non-real; \\
 $q >0, s=\frac{q^2}{4}, r=0$, two pairs of equal non-real zeros ;\\
 $q >0, s=0$, two equal real zeros, two non-real;\\
 \item $q =0, s>0,$ two equal real zeros, two non-real; \\
       $q =0, s=0$, four equal real zeros.\\
\end{enumerate}
\end{enumerate}
\end{theorem}

For concrete application of Theorem \ref{lovence} in this work, we consider 
 $G(z)$ in \eqref{Flav1x} defined by
\begin{eqnarray*}\label{katiki} G(z)=3 \gamma z^4-2\beta z^3+ \alpha z^2 +1, \end{eqnarray*}
from which  we obtain 
\begin{eqnarray}\label{Flav2}
J(u)=u^4+ \alpha u^2-2\beta u +3 \gamma.
\end{eqnarray}
Equation (\ref{Flav2}) has the same form as (\ref{Flav22}) in  Theorem \ref{lovence}. We now  apply this theorem to $J(u)$ subject to  $q=\alpha, r=-2 \beta, s=3\gamma$. For a possible formation of a\textit{ smooth Jordan curve with $0$ in its interior}, it is a necessary condition that $G(z)$ has at least two simple real zeros of opposite signs. Thus it is enough  to consider only 3 possible cases  of zeros of $J(u)$, namely:
\begin{enumerate}[(S1)]\item
The polynomial  $J$ has two distinct real zeros of opposite signs and a pair of complex conjugate zeros. 
\item All the zeros of $J$ are real and distinct. 
 \item The polynomial $J$ has all real zeros in which two and only two are equal. 
\end{enumerate}
 
\subsection{The connected components of real quartic polynomials} \noindent 

\medskip
Consider the family of real quartic  polynomials of the form 
\begin{eqnarray}\label{swallowtail}
\mathcal{F}= t^4+c_1t^2+c_2t+c_3.
\end{eqnarray}
Then $\mathcal{F}$ can be identified with the space of coefficients $(c_1, c_2, c_3)$ also called the parameter space.  We define the \textit{standard discriminant surface} denoted by $\mathscr{D}(\mathcal{F})$ as the subset of the parameter space for which the associated polynomial has multiple real zeros.

From \eqref{swallowtail}, we obtain
\begin{eqnarray*}\label{Bbba}
\bigtriangleup (\mathcal{F})=-4 c_1^3 c_2^2 - 27 c_2^4 + 16 c_1^4 c_3 + 144 c_1 c_2^2 c_3 - 
 128 c_1^2 c_3^2 + 256 c_3^3.
\end{eqnarray*}
and
\begin{eqnarray*}
{\small\mathscr{D}(\mathcal{F})=\Bigl\{ (c_2,c_3,c_4):-4 c_1^3 c_2^2 - 27 c_2^4 + 16 c_1^4 c_3 + 144 c_1 c_2^2 c_3 - 128 c_1^2 c_3^2 + 256 c_3^3= 0 \Bigl\}} .
\end{eqnarray*}
The standard discriminant surface $\mathscr{D}(\mathcal{F})$ for \eqref{swallowtail}
is also called a swallowtail, see Fig.1.
\begin{figure}[!htb] \label{Nabs}
  \centering 
   \includegraphics[height=7cm, width=12.5cm ]{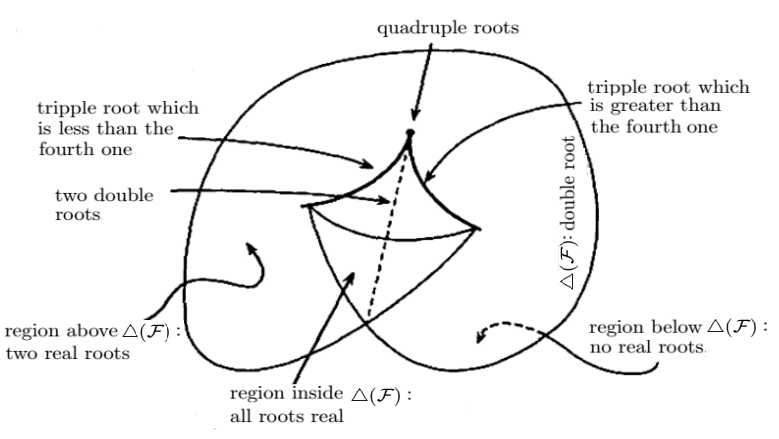}
    \caption{The discriminant surface $\mathscr{D}(\mathcal{F})$ of \eqref{swallowtail}}
    Source: Discriminants, resultants, and multidimensional determinants \cite{Ge}, pg 381. \end{figure}
  It is a surface which self-intersects along a curve corresponding to polynomials having two double real zeros, the cuspidal cubic  edges correspond to the polynomials with one triple zero while the most singular point is the quartic polynomial $x^4$ with a zero of multiplicity $4$. For the spaces of polynomials of higher degrees, a similar strata will also appear \cite{VVA}.
For a given  $c=(c_1, c_2, c_3)\in \mathbb{R}^3$, one can determine in which region of the complement of $\mathscr{D}(\mathcal{F})$ the point $c$ lies, by checking if the set of real zeros of the polynomial has $0, 2$ or $4$ elements:- these being the only possibilities for a real quartic polynomial, as non-real zeros appear in pairs as complex conjugates.

Our main interest is the space of all polynomials in $\mathcal{F}$ with a non-vanishing discriminant. This amounts to studying the connected components which form the complement of the discriminant surface in the parameter space of $\mathcal{F}$. In our case, such space will consist of polynomials without multiple real zeros.

The complement of the discriminant surface in the parameter space consists of several connected components. It is a finite disjoint union of connected open sets, also known as \textbf{cells}, such that the number of real zeros of polynomials in $\mathcal{F}$ does not change when the parameters vary within the same cell. In other words, the polynomials lying in the same cell have diffeotopic real zero sets \cite{LE}.
The discriminant variety forms the boundary between these cells, such that the number of real zeros of polynomials change only when crossing a boundary \cite{BBs,YBF}. 

Let us consider a familiar example of a  family of real quadratic polynomials $f(x)=ax^2+bx+c$. The corresponding discriminant is $\bigtriangleup(f) = b^2 −4ac$. The sign of  $\bigtriangleup(f) $ tells us how many real roots a quadratic polynomial has. In other words, it tells us about the topology of the real zero set. The same idea holds true for polynomials with degree $> 2$ and in particular \eqref{swallowtail}.
\medskip

In the final part of this section, we state the following lemmata which will be used in the proof of the main results. For their proofs, refer to \cite{BSF}.
\begin{lemma} \cite{BSF}\label{lem}  Suppose $f(z)$ is an analytic function in a neighborhood of a bounded non-empty region $\Phi \subset \mathbb{C}$ such that $\Im (f(z)) = 0$ for all $z \in \partial\Phi$. Then $f(z)$ is constant on $\Phi$.
\end{lemma}
\begin{lemma}  \cite{BSF}\label{Agn}
 Let $f(z)$ be an analytic function in a neighborhood of a disk  $\textit{D} \subset \mathbb{C}$. Define
$\Gamma:= \{z \in \textit{D} : \Im (f(z)) = 0\}$ and suppose that $\Gamma \neq \emptyset$ and $\textit{D} \setminus \Gamma$ is connected. Then $f(z)$ is constant on $\textit{D}$.
\end{lemma}
From now onwards, we shall use  $b(z)=- \frac{1}{z}+ \alpha z  - \beta z^2+ \gamma z^3$, the real rational function defined in  \eqref{shamim22}, and 
$\Gamma_b:=b^{-1}(\mathbb{R}P^1)= b^{-1}(\mathbb{R}) \cup \{0, \infty\}. $ 
It is clear that  $b^{-1}(\mathbb{R}) \subset  \mathbb{C}\setminus \{0 \}$.

\section{The Laurent polynomial with all real critical points}\label{sect4}
In this section, we prove the first part of Theorem 2 which is a situation where  all the critical points of $b(z)$ are real i.e., $G(z)$ defined in \eqref{Flav1x} is hyperbolic. We first discuss cases $(S2)$ and $(S3)$ corresponding to the reality of zeros of $G(z)$ for which  $b^{-1}(\mathbb{R})$ contains a smooth Jordan curve with $0$ in its interior. Later, we handle the case where $G(z)$ is hyperbolic but  $b^{-1}(\mathbb{R})$  contains a non-simple curve with $0$ in its interior. 

\begin{lemma} \label{only}
In the above notation, consider the recurrence \eqref{christ22} subject to \eqref{almst}. Then, $P_2(z)$ is hyperbolic if and only if $\alpha \leq 0$.
\end{lemma}
\begin{proof}
For $n=2$, substituting the initial conditions  in \eqref{christ22} gives $P_2(z)=z^2+ \alpha,$ 
which is clearly  real-rooted if and only if $\alpha \leq 0$.
\end{proof}

\begin{lemma}\label{HHABB}
Let  $Q(z)=az^4+bz^3+cz^2+1$ be a real quartic hyperbolic  polynomial. Then at least two of the zeros of $Q(z)$ have different signs.
\end{lemma}

\begin{proof}
Suppose that $z_i$ for $1\leq i\leq 4$ are the real zeros of $Q(z).$ Then none of the $z_i$ is $0$ since $Q(0)\neq 0$. Now, define a new polynomial $M(\tau)=\tau^4+c\tau^2+b \tau + a$ whose zeros are $\tau_i$ for $1\leq i\leq 4$. Since $\sum_{i=1}^{4}\tau_i=0$, then $\tau_i\tau_j <0$ for some $i, j \in \{1,2,3,4 \}$ where $i \neq j$. The result follows from the fact that for each  $i$, the zeros of $M$ are related to the zeros of $Q$ by $\tau_i=\frac{1}{z_i}$. 
\end{proof}

\begin{lemma}\label{SHAM}
Let $b(z)$ and $G(z)$ be defined as in  \eqref{shamim22} and \eqref{Flav1x} respectively. If $G(z)$ has all distinct real zeros, then $b^{-1}(\mathbb{R})$ contains  a Jordan curve with $0$ in its interior. 
\end{lemma}

\begin{proof} Suppose $G(z)$ has all real and distinct zeros $z_i$ for $1\leq i\leq 4$. As seen above, none of the zeros of $G(z)$ is zero. 
Let $z_3$ and $z_4$ be the zeros of $G(z)$ closest to the origin and having different signs. Since all the zeros of $G(z)$ are distinct, then for each $i$, $z_i$ is a critical point of $b(z)$ of order $1$. Therefore locally in the neighbourhood of each $z_i$, we have  $b(z)\sim z^2.$
Let $w_i$ be the real critical value with $z_i$ as the corresponding critical point. For a sufficiently small $\epsilon >0$, let  $V_i=(w_i -\epsilon, w_i +\epsilon) \subset \mathbb{R}$.  Then  $U_i= b^{-1}(V_i)$, the pre-image of $V_i$ consists of two distinct curves (arcs) intersecting only at $z_i$. Of these, one is a line segment on the real line while the other is a non-real curve. 

Concerning the point at infinity, observe that since $z=\infty$ is a critical point of $b(z)$ of order $2$, we have that $b(z) \sim \gamma z^3$ as $z \to \infty$, and so, locally three curves in $\Gamma_b$ pass through $\infty$, one of which is the real axis and the remaining two are non-real curves. By Lemma \ref{Agn}, any curve must start and terminate only at critical points. The conclusion is that the two non-real curves through infinity pass through  $z_1$ and $z_2$ respectively and the non-real curve through $z_3$ necessarily closes through $z_4$ into a Jordan curve enclosing $0$. By Lemma \ref{lem}, exactly one  Jordan curve can be formed in  $ b^{-1}(\mathbb{R})$.  

Furthermore, no curve connects two distinct  critical points $z_i$ and $z_j$ of the same sign. This is because if such a curve $\Psi$ existed, then $b(z)$  would be analytic in the region $D$ whose boundary is $\Psi \cup I$ where $I \subset \mathbb{R}$ is the segment on the real axis that joins  $z_i$ and $z_j$. Since $\Im(b(z))=0$  for all points $z$ on the said boundary, Lemma \ref{lem} shows that $b(z)$ is constant on $D$, a contradiction. Hence the Jordan curve  enclosing $0$ is uniquely formed in $b^{-1}(\mathbb{R})$.  
\end{proof}

\begin{lemma}\label{SHAMMALE}
In the notation of Lemma \ref{SHAM}, let $G(z)$ be hyperbolic with one double real root and two simple roots.  Then $b^{-1}(\mathbb{R})$ contains  a Jordan curve with $0$ in its interior if and only if  the multiple root is larger in absolute value than the simple root of the same sign.
\end{lemma}
\begin{proof}
Let $z_1, z_2, z_3, z_4$ be real zeros of $G(z)$ with $z_2=z_3$. By Lemma \ref{HHABB}, let $z_1$ and $z_4$ have different signs. Then $z_1, z_2, z_3$ or $z_2, z_3, z_4$ have the  same sign.  We choose the first case and by symmetry the other case works similarly.  Assume that $|z_1|< |z_2|.$ Since $z_1$ and $z_4$ are simple critical points of $b(z)$, locally at these points, $b(z)$ consists of a line segment on the real line intersecting a non-real curve only once. Since $z_2$ is a real critical point of $b(z)$ of order $2$, locally  $b(z)$ consists of a real line intersecting two non-real curves only at $z_2$.  The the point $z=\infty$ has the same order as $z_2$, and so locally  $b(z)$ behaves similarly at these points.
By Lemma \ref{lem}, no curve through $z_1$ passes through $z_2$ since $b(z)$ is a non-constant  function. In addition, if the two curves through $z_2$ pass through $\infty$ and $z_4$ respectively, this implies  that a curve through $z_1$ passes through $\infty$. However this means that the curve connecting  $z_1$ to $\infty$ and the one connecting $z_2$ to $z_4$ intersect outside the real axis, creating a non-real critical point of $b(z)$, contradicting the fact that $G(z)$ is real-rooted.  
 The conclusion is that the two complex curves through $z_2$ pass through $\infty$, and the curve through  $z_1$  necessarily closes through $z_4$ into a Jordan curve enclosing $0$.  See for example Fig. \ref{fig4}. This completes the proof. 
\end{proof}

\begin{remark} \label{RMK}
\end{remark}
\begin{enumerate}
[(a)]\item
In Lemma \ref{SHAMMALE}, if  the multiple root is smaller in absolute value than the simple zero of the same sign i.e., $|z_2|<|z_1|$, then using a similar argument as in the previous proof, we obtain that $b^{-1}(\mathbb{R})$ contains a non-simple curve containing $0$ in the interior, formed as follows: Two curves pass through $z_2$, of which one closes back through $z_4$ to form a non-simple curve enclosing $0$ while the other connects $z_2$  to $\infty$ since 
by Lemma \ref{lem}, (see Fig. \ref{fig5}).

\item Using similar arguments, if  $z_2=z_3$ has a different sign from that of $z_1$ and $z_4$,  we obtain the same results as in Remark \ref{RMK}$(a)$, (see Fig. \ref{fig6}).

\item A similar argument holds when $G(z)$ has two pairs of equal zeros and when $G(z)$ is hyperbolic with three equal zeros. (See Fig. \ref{fig7} and Fig. \ref{fig8} respectively).
\end{enumerate}
\begin{figure}[!htb]
   \begin{minipage}{0.3\textwidth}
     \centering
     \includegraphics[height=4cm, width=4cm]{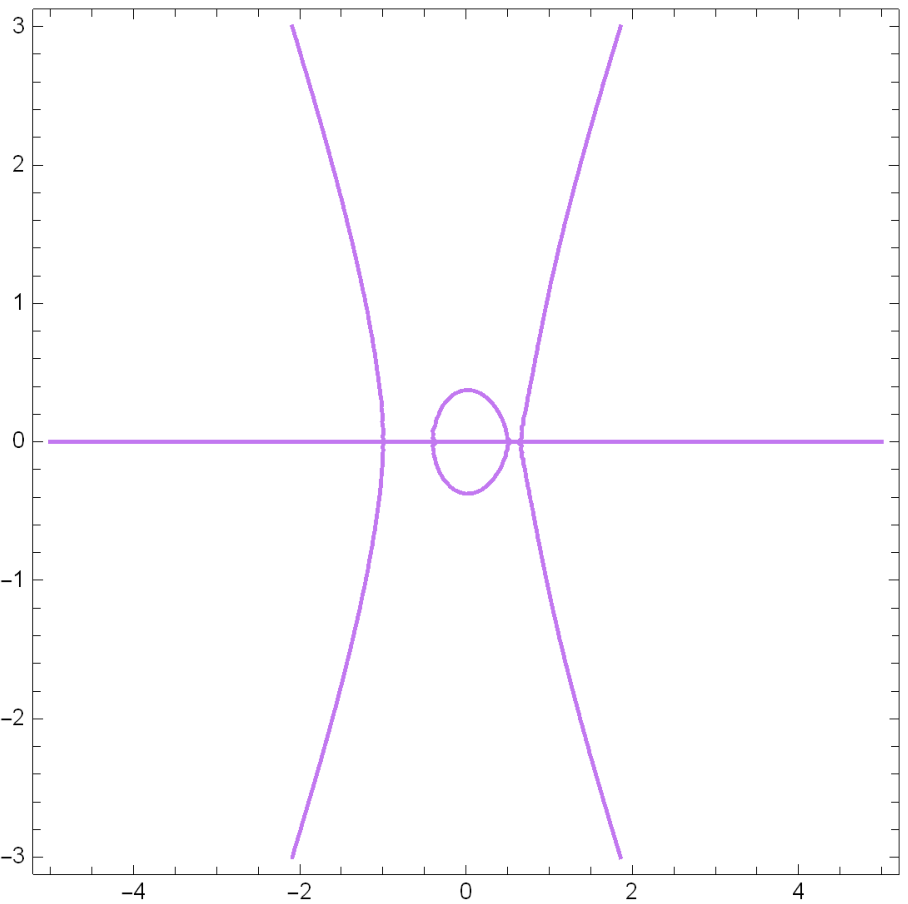}
     \caption{}\label{fig3}
   \end{minipage}\hfill
   \begin{minipage}{0.3\textwidth}
    \centering
     \includegraphics[height=4 cm, width=4cm]{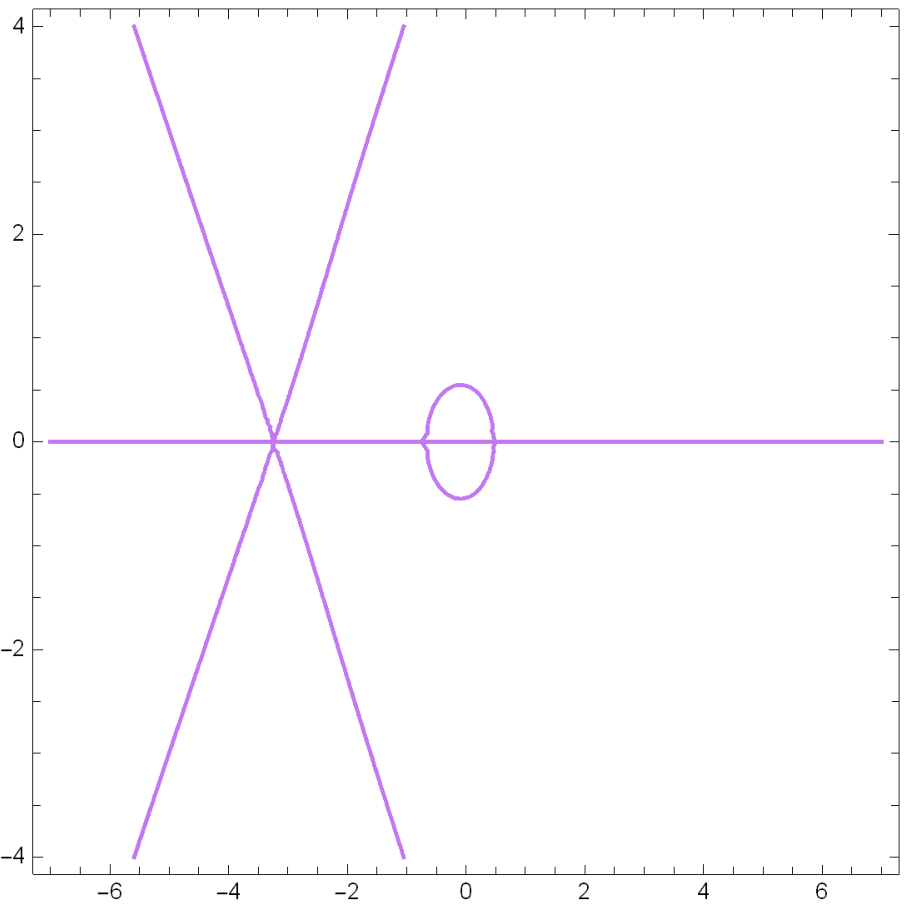}
    \caption{}\label{fig4}
   \end{minipage}\hfill
   \begin{minipage}{0.3\textwidth} 
     \centering
     \includegraphics[height=4cm, width=4cm]{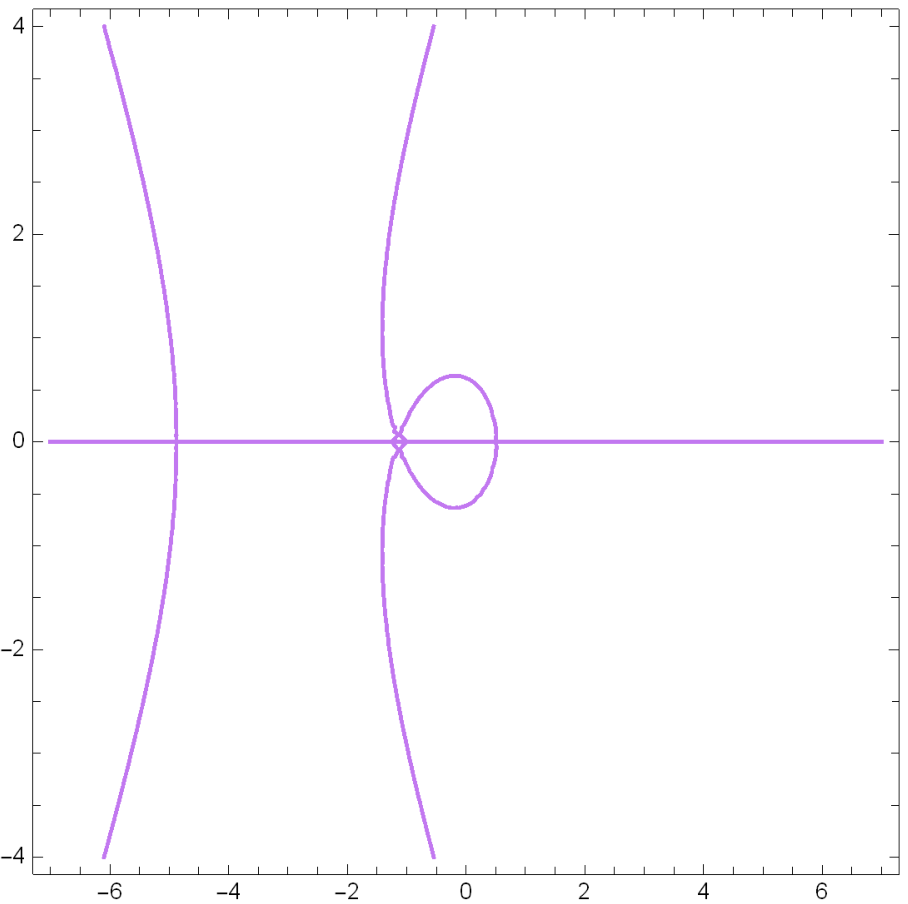}
     \caption{}\label{fig5}
   \end{minipage}
 \end{figure}
\medskip
  \begin{figure}[!htb]
   \begin{minipage}{0.3\textwidth}
     \centering
     \includegraphics[height=4cm, width=4cm]{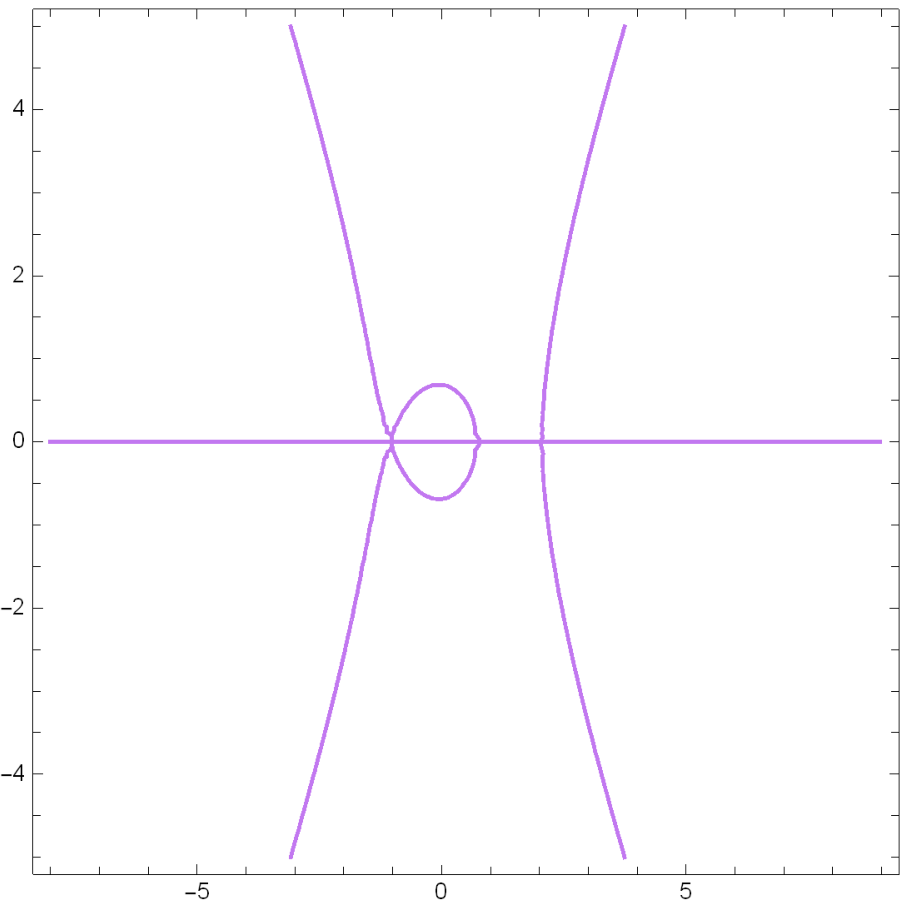}
     \caption{}\label{fig6}
   \end{minipage}\hfill
   \begin{minipage}{0.3\textwidth}
    \centering
     \includegraphics[height=4cm, width=4cm]{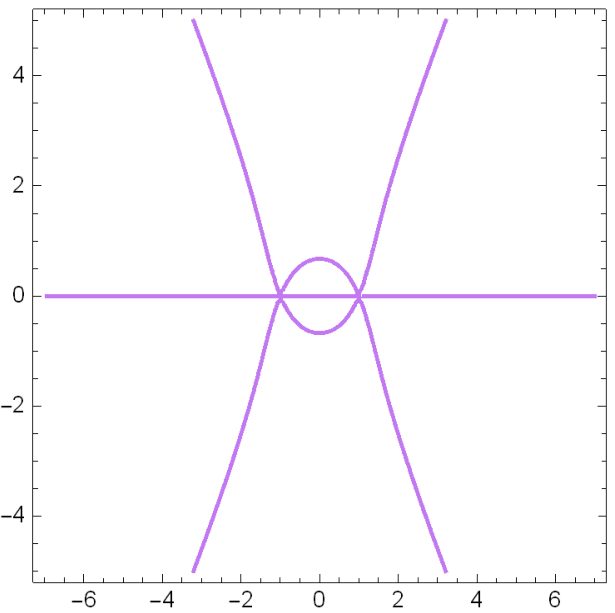}
    \caption{}\label{fig7}
   \end{minipage}\hfill
   \begin{minipage}{0.3\textwidth}
     \centering
     \includegraphics[height=4cm, width=4cm]{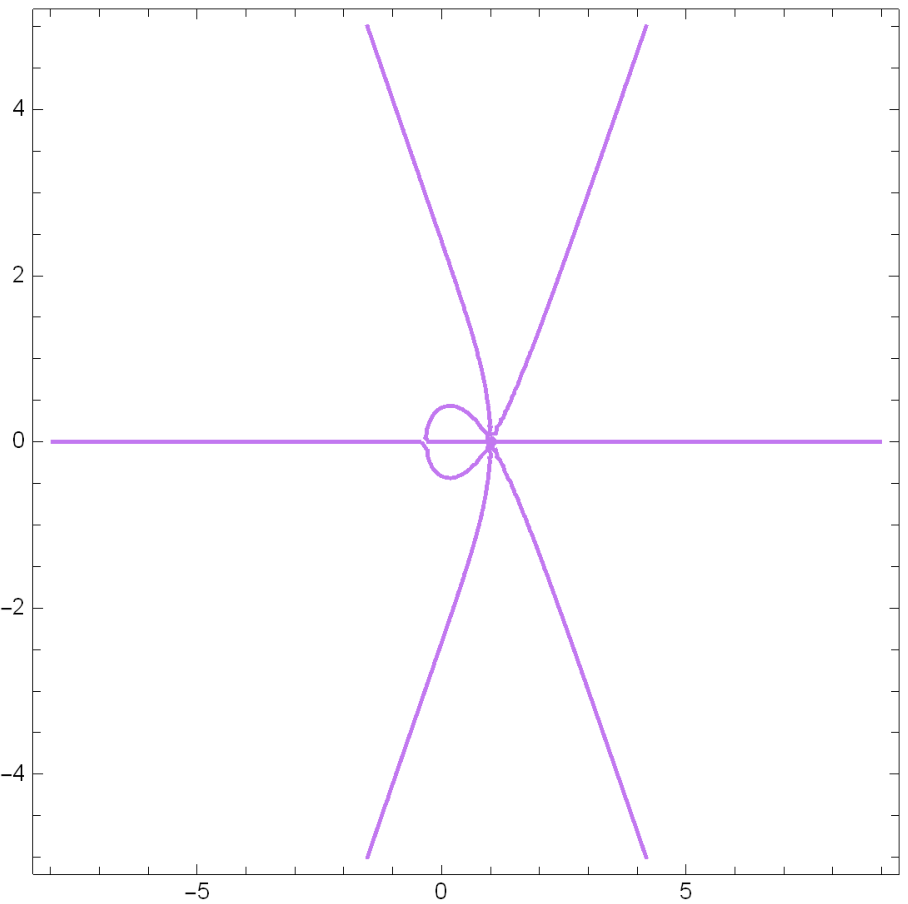}
     \caption{}\label{fig8}
   \end{minipage}
 \end{figure}
 
\medskip
  We now settle  Theorem 2 (a):
   \medskip
   
  \textit{Proof of Theorem 2 (a)}.
Suppose that $G(z)$ is hyperbolic with either distinct zeros or with one double real zero
whose absolute value is larger than that of a simple zero of the same sign. Then Lemmas \ref{SHAM} and \ref{SHAMMALE} show that  in either case, $b^{-1}(\mathbb{R})$ contains a smooth Jordan curve with $0$ in its interior. Hence Theorem $1(ii)$ is satisfied.
By the implication $(ii)\Rightarrow (iii)$ in Theorem $1$, we deduce  that all the zeros of $P_n(z)$ are real for all positive integer $n$.

\medskip
The cases where the hyperbolic polynomial $G(z)$ has repeated zeros are covered in Remark \ref{RMK} where it is shown that $b^{-1}(\mathbb{R})$ contains a non-simple curve with $0$ in the interior. 
Note that for each positive integer $n$, the zeros of $P_n(z)$ will continuously depend on the parameters $\alpha, \beta$ and $\gamma$. Now, since the zeros of $P_n(z)$ are real when the zeros of $G(z)$ are real and distinct, they will remain real in the limit when $G(z)$ has repeated zeros. 

We note that the condition that $G(z)$ is hyperbolic is equivalent to  \begin{eqnarray*}
\alpha<0,~~~~~ -\frac{\alpha^2}{36} \leq \gamma \leq  \frac{\alpha^2}{12}~~\mbox{and}~ \bigtriangleup(G)\geq 0 \end{eqnarray*}
 follows immediately from applying  Theorem \ref{lovence}$(b)~~\&~~ c(i)$ to $J(u)$ (and consequently to $G(z)$). This completes the first part of Theorem $2$.

\section{The Laurent polynomial with two real and a pair of complex conjugate critical points } \label{sect5}

In this section we prove the second part of Theorem $2$ where  $b(z)$ has two distinct real critical points and a pair of complex conjugate critical points. Our main task is to determine when  $b^{-1}(\mathbb{R})$ contains  a Jordan curve enclosing $0$.  

\begin{proposition} Suppose that $f(z) = f(x + iy) = u(x, y) + iv(x, y)$ is differentiable at
$z_0 = x_0 + iy_0$. Then the partial derivatives of $u$ and $v$ exist at $z_0 = (x_0, y_0)$ and
\begin{eqnarray}\label{eddy}
 f'(z_0) = \frac{\partial u}{\partial x}(x_0, y_0) + i\frac{\partial v}{\partial x}(x_0, y_0) = \frac{\partial v}{\partial y}(x_0, y_0) - i\frac{\partial u}{\partial y}(x_0, y_0). \end{eqnarray}
Thus we obtain the Cauchy-Riemann equations 
\begin{eqnarray}\label{eddy2}
\frac{\partial u}{\partial x}=\frac{\partial v}{\partial y}~~ \mbox{and}~~ \frac{\partial u}{\partial y}=-\frac{\partial v}{\partial x}.
\end{eqnarray}  
  \end{proposition}
From \eqref{eddy} and \eqref{eddy2}, we have $f'(z_0)=0$ if and only if for $z_0=(x_0, y_0)$,
  \begin{eqnarray} \label{Rest1}
 \frac{\partial u}{\partial x}(x_0, y_0) =0 =\frac{\partial v}{\partial x}(x_0, y_0)=\frac{\partial u}{\partial y}(x_0, y_0)=\frac{\partial v}{\partial y}(x_0, y_0). \end{eqnarray}

\begin{lemma}\label{HHABxxB}
Let  $Q(z)= \lambda_1 z^4+\lambda_2z^3+\lambda_3z^2+1$ be a real quartic univariate polynomial with two real zeros and a pair of complex conjugate zeros.  The real zeros of $Q(z)$ have different signs if and only if  $\lambda_1<0$.
\end{lemma}
\begin{proof}
Let $a,b,c+id, c-id$  be the zeros of $Q(z)$, where $a,b,c,d \in \mathbb{R}, d \neq 0$ and  $\lambda_1<0.$ All the roots of $Q(z)$ are  non-zero since $Q(0)=1$. Note that
\begin{equation*} \label{eq11}
\begin{split}
Q(z) & = (z-a)(z-b)(z-(c+id))(z-(c-id)) \\
 & = z^4-(a + b + 2 c) z^3 + (a b + 2 a c + 2 b c + c^2 + d^2) z^2 -\\ & \quad(2 a b c + a c^2 + b c^2 + a d^2 + b d^2) z+ a b c^2 + a b d^2.
\end{split}
\end{equation*}
Comparison of the coefficients of the constant term gives
\begin{eqnarray}\label{HAB1222}
ab(c^2+d^2)=1/\lambda_1.
\end{eqnarray}
Since $(c^2+d^2)>0$ and $\lambda_1<0$, then Equation \eqref{HAB1222} holds true if and only if $ab<0$, which proves the claim.
\end{proof}

\subsection{Singular points of a net} 

Let us establish the conditions guaranteeing the location of the non-real critical points of $b(z)$  on the curve $\Gamma_b$. 
 We begin with the following basic definition. %
 \begin{definition}
 An affine algebraic plane curve over $\mathbb{R}$ or a real affine algebraic plane
curve is a subset  of $\mathbb{C}^2$ of the form $\mathcal{C}=\{(x, y) \in \mathbb{C}^2: F(x, y) = 0\}$ for some polynomial $F(x,y) \in \mathbb{R}[x, y]$.
 The real part of this curve is its intersection with  $\mathbb{R}^2$, i.e., the set of real solutions of $F(x,y) \in \mathbb{R}[x, y]$.  \end{definition}

\begin{definition}\label{label} (see \cite{Walk})
Let $\mathcal{C}$ be an affine  algebraic curve over $\mathbb{R}$ defined by
$F(x,y) \in \mathbb{R}[x, y]$ and $p = ({x^{*}},y^*) \in\mathcal{C}$. Then $p$ is a point of multiplicity $m$ on $\mathcal{C}$ if and only if all the partial derivatives of $F(x,y)$ up to and including order $(m-1)$ vanish at $p$, but at least one partial derivative of order $m$ does not vanish at $p$.
A point of multiplicity two or more is said to be a singular point, in particular the point of multiplicity two is called a double point. A point of multiplicity one is called a simple point.  If there is at least one singular point on the curve, we say that $\mathcal{C}$ is \textit{ a singular curve}. (Otherwise $\mathcal{C}$ is called smooth or nonsingular.)
\end{definition}
Therefore the necessary and sufficient condition for  $p  \in \mathcal{C}$ to be singular is  that
\begin{eqnarray}\label{LL}
\frac{\partial F}{\partial x}(p)=\frac{\partial F}{\partial y}(p)=0.
\end{eqnarray}
\medskip

The following lemma is a reformulation of the one found in  \cite{IP}. 
\begin{lemma}\label{kalu} 
Let $\mathcal{C}$ be a real algebraic curve given by $F(x, y) = 0$. If the coefficients of $F$ are continuously varied, the topology of the real zero set of  $F(x, y)$ changes only when the coefficients pass through the values for which $\mathcal{C}$ is singular.
\end{lemma}
 
Now, with $b(z)= - \frac{1}{z}+ \alpha z  - \beta z^2 + \gamma z^3$ and $z \in \mathbb{C}$, we have  that $z \in b^{−1}(\mathbb{R})$ if and only if $\Im (b(z)) = 0$, and this latter condition can be turned into a polynomial equation $R(x, y) = 0,$ where $R \in \mathbb{R}[x, y],~~ x = \Re (z)$ and $y = \Im(z).$ It can be seen from \eqref{Rest1} that the singular points of the curve  $b^{-1}(\mathbb{R})$ are the critical points of $b(z)$.  Note that
\begin{eqnarray*}\label{wed12}
R(x,y)=y + \alpha x^2 y -2 \beta x^3 y + 3 \gamma x^4 y + \alpha y^3 - 2 \beta x y^3 + 2 \gamma x^2 y^3 - \gamma y^5,
\end{eqnarray*}
and is divisible by some power of $y$ since $R(x,0) = 0$. In our case, we have
\begin{eqnarray}\label{wed31X}
R(x,y) = y S(x,y)
\end{eqnarray}
where 
\begin{eqnarray}\label{wed312}
S(x,y)=1 + \alpha x^2 - 2 \beta x^3 +3 \gamma x^4 + \alpha y^2 - 2 \beta x y^2 + 2 \gamma x^2 y^2 - \gamma y^4
\end{eqnarray}
 and $S(x,0) \neq 0$.
 
 \begin{lemma} \label{Kiwede}
Let $b(z), G(z)$ and $S(x,y)$ be defined as above and $\bigtriangleup(G)$ be the discriminant of $G(z)$. Furthermore, let $\mathcal{C}$ be the curve defined by $S(x,y)=0$ and  $p=(x^{*},y^*)\in \mathcal{C}$.
Then $p$ is a singular point of $\mathcal{C}$ if and only if  either
\begin{eqnarray*}
3 \beta ^2 - 8\alpha \gamma \leq 0 ~~~~\mbox{and}~~~~  16 \gamma^2 \left(\alpha^2+4 \gamma\right)-8 \alpha \beta ^2 \gamma+\beta ^4=0
\end{eqnarray*}
or
\begin{eqnarray*}
3 \beta ^2 - 8\alpha \gamma \geq 0 ~~~~\mbox{and}~~~~ \bigtriangleup(G)=0.
\end{eqnarray*}
\end{lemma}
\begin{proof}

By \eqref{LL},  $p=(x^{*},y^*)\in \mathbb{R}^2$ is a singular point of $R(x,y)$ if and only if 
\begin{eqnarray}\label{uuu}
R({x^{*}},y^*) = \frac{\partial R}{\partial x}({x^{*}},y^*) = \frac{\partial R}{\partial y}({x^{*}},y^*)=0.\end{eqnarray}  
By \eqref{wed31X}, we obtain that \eqref{uuu} is equivalent to
\begin{eqnarray*}\label{eq12}
y^*S({x^{*}},y^*) = y^* \frac{\partial S}{\partial x}({x^{*}},y^*) = S({x^{*}},y^*)+
y^*\frac{\partial S}{\partial y}({x^{*}},y^*)=0.\end{eqnarray*}  
So either $y^*= 0$ or $y^* \neq 0$ and $
S({x^{*}},y^*) = \frac{\partial S}{\partial x}({x^{*}},y^*) = 
\frac{\partial S}{\partial y}({x^{*}},y^*)=0.$ \\
Observe that
\begin{eqnarray}\label{LL1}
\frac{\partial S}{\partial x}({x},y)&=&2 \alpha {x} - 6 \beta {x}^2 + 12 \gamma {x}^3 - 2 \beta {y}^2 + 4 \gamma {x} {y}^2, \\\label{xc}
\frac{\partial S}{\partial y}({x},y)&=&2 \alpha y - 4 \beta xy + 4 x^2y  -4\gamma {y}^3.
\end{eqnarray}
Solving \eqref{LL1} and \eqref{xc} and using the fact that $\gamma \neq 0$, yields $p=(x^{*},y^*)$ as a singular point of 
 $S$  if and only if either
\begin{eqnarray*}
({x^{*}},y^*)=\left( \frac{\beta}{4 \gamma},\pm \frac{\sqrt{-3 \beta ^2 + 8\alpha \gamma}}{4\gamma}\right)~~\mbox{and}~~ 64 \gamma^3 +(4\alpha\gamma-\beta^2)^2=0
\end{eqnarray*}
or
\begin{eqnarray*}({x^{*}},y^*)&=&\left(\frac{3 \beta \pm \sqrt{3} \sqrt{3 \beta^2-8\alpha \gamma}}{12\gamma}, 0\right)~~\mbox{and}  \end{eqnarray*}
\begin{eqnarray*}
-192 \gamma^3(3\alpha^4 \gamma-\alpha^3 \beta^2-72 \alpha^2 \gamma^2+108 \alpha \beta^2 \gamma-27 \beta^4+432 \gamma^3)= -192 \gamma^3 \bigtriangleup(G)=0.
 \end{eqnarray*}
\end{proof}

\begin{lemma}\label{SH}
Let  $b(z), G(z)$and $H(z)$ be  defined in \eqref{shamim22}, \eqref{Flav1x} and \eqref{Auntkaduka} respectively. Further, assume that $\gamma <0$ and $\bigtriangleup(G)<0$.
Then $b^{-1}(\mathbb{R})$ contains a smooth Jordan curve with $0$ in its interior if and only if $\bigtriangleup(H)>0$ and $12\alpha \gamma-\beta^2 \geq 0$. 
\end{lemma}
\begin{proof}
 By Theorem \ref{lovence}, the condition $\bigtriangleup(G)<0$ is equivalent to  $G(z)$ having two  real zeros and a pair of complex conjugate zeros. Moreover, for $\gamma<0$, the real zeros of $G(z)$ have different signs by Remark \ref{HHABxxB}. \\
Consider a domain  $\Omega$ in $\mathbb{R}^3$ given by
$$\Omega=\{ (\alpha, \beta, \gamma)\in \mathbb{R}^3: \gamma<0 ~~\mbox{and} \bigtriangleup(G)<0 \}.$$
Define  $S(x,y)$ by \eqref{wed312}, and let $\mathcal{C}$ be the curve given by $S(x,y)=0.$

By Lemma \ref{Kiwede}, it follows that $\mathcal{C}$ in $\Omega$ is non-singular  if and only if \begin{eqnarray*}
 \bigtriangleup(H)=64 \gamma^3 +(4\alpha\gamma-\beta^2)^2=0.
\end{eqnarray*} 
We seek to understand  how $\Omega$ is splitted by the surface $\bigtriangleup(H)=0.$ Since $\bigtriangleup(G)$ and $\bigtriangleup(H)$ are weighted homogeneous polynomials with the same weights in $ (\alpha, \beta, \gamma)$ given by $w(\alpha)=2, w(\beta)=3$ and  $w(\gamma)=4$, in order to understand the topology of  $\Omega\setminus \{ \bigtriangleup(H)=0\}$, it suffices to consider the restriction of $\Omega$ and $ \bigtriangleup(H)=0$ to some affine plane and take a cone over this restriction. In our situation, it is convenient to take $\beta=1$.\\
By abuse of notation, we refer to the restriction of $\Omega,\bigtriangleup(H)$ and $\bigtriangleup(G)$  to $\beta=1$ also by $\Omega,\bigtriangleup(H)$ and $\bigtriangleup(G)$.
Using change of variables, set $v=\alpha\gamma-\frac{\beta^2}{4}$ (so that $v=\alpha\gamma-\frac{1}{4}$ in the plane $\beta=1$). The map $(\gamma,\alpha )\mapsto(\gamma, \alpha\gamma-\frac{\beta^2}{4})$ is a global diffeomorphism of 
$\mathbb{R}_{<0}\times \mathbb{R}$ to $\mathbb{R}_{<0}\times \mathbb{R}$. 
Therefore the topology of $\mathcal{C}$ and its complement in $\mathbb{R}_{<0}\times \mathbb{R}$ is the same in both $\gamma\alpha$- and  $\gamma v$-planes.\\
 Thus  $\bigtriangleup(H)=0$ and $\bigtriangleup(G)=0$  become $4\gamma^3+v^2=0$ and
\begin{eqnarray*}\label{barbie}
1+96(12\gamma^3-v^2)-768(-12\gamma^3+v^2)^2-512v(36\gamma^3+v^2)=0 \end{eqnarray*} 
respectively, (see Fig. \ref{foc4}).
Observe that $\bigtriangleup(G)>0$ only in $\Omega_4$, thus $\Omega$ consists precisely of 3 connected components namely $\Omega_1, \Omega_2$ and $\Omega_3$ with $\bigtriangleup(H)>0, \bigtriangleup(H)<0$ and $\bigtriangleup(H)>0$ respectively.
\begin{figure} [hbt!]
   \centering
     \includegraphics[height=5cm, width=8cm]{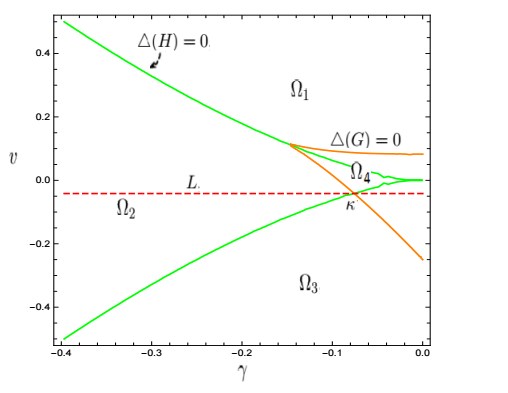}
 \caption{ }\label{foc4}
 \end{figure}
 
\textit{Let $L$ be the line in $\Omega$ given by  $v=-\frac{1}{24}$. Then $L$ lies neither  in $\Omega_1$ nor in $\Omega_3$ and can be used to distinguish between  $\Omega_1$ and $\Omega_3$ as follows:}\\
Let $\gamma \in I_1$ where $I_1= (-\infty, c)$ with $c=-\frac{1}{4(\sqrt[3]{6})^2}$. Then $\bigtriangleup(H)$ and $\bigtriangleup(G)$ vanish at $c$.
Note that  $\bigtriangleup(H)$ subject to $L$ gives a cubic polynomial $q(\gamma)=4 \gamma^3+\frac{1}{576}$. For all $\gamma\in I_1$, we have that $q(\gamma)<0$ and hence  $L$  is entirely contained in  $\Omega_2$.\\
Next, we consider the interval $I_2= (c, 0)$. Then $\bigtriangleup(G)$ subject to $L$ gives $\widetilde{q}(\gamma)=-110592\gamma^6+1952 \gamma^3+\frac{125}{144}$. For all $\gamma\in I_2$, we obtain $\widetilde{q}(\gamma)>0$ and hence  $L$  is entirely contained in  $\Omega_4$. Hence  $\Omega_1$ lies in the subdomain of $\Omega$ where $L\geq 0$  while   $\Omega_3$ lies in the subdomain of $\Omega$ where $L<0$. 

Therefore we can explicitly define $\Omega_1,\Omega_2$ and $\Omega_3$ as follows:
\begin{eqnarray*} 
\Omega_1&=& \Big\{(\gamma, v) \in \Omega : \bigtriangleup(H)>0~~\mbox{and}~~L\geq 0 \Big\},\\
\Omega_2&=& \Big\{(\gamma, v) \in \Omega : \bigtriangleup(G)<0 \Big\},\\
\Omega_3&=& \Big\{(\gamma, v) \in \Omega : \bigtriangleup(H)>0~~\mbox{and}~~L<0\Big\}.  
\end{eqnarray*}
By Lemma \ref{kalu}, it suffices to pick an arbitrary point in each of $\Omega_1, \Omega_2, \Omega_3$, and then check the topology of $\mathcal{C}$ (see Figures \ref{foc7},\ref{foc8} and  \ref{foc9a} respectively using Mathematica). 
\begin{figure}[!htb]
   \begin{minipage}{0.3\textwidth}
     \centering
     \includegraphics[height=4cm, width=4cm]{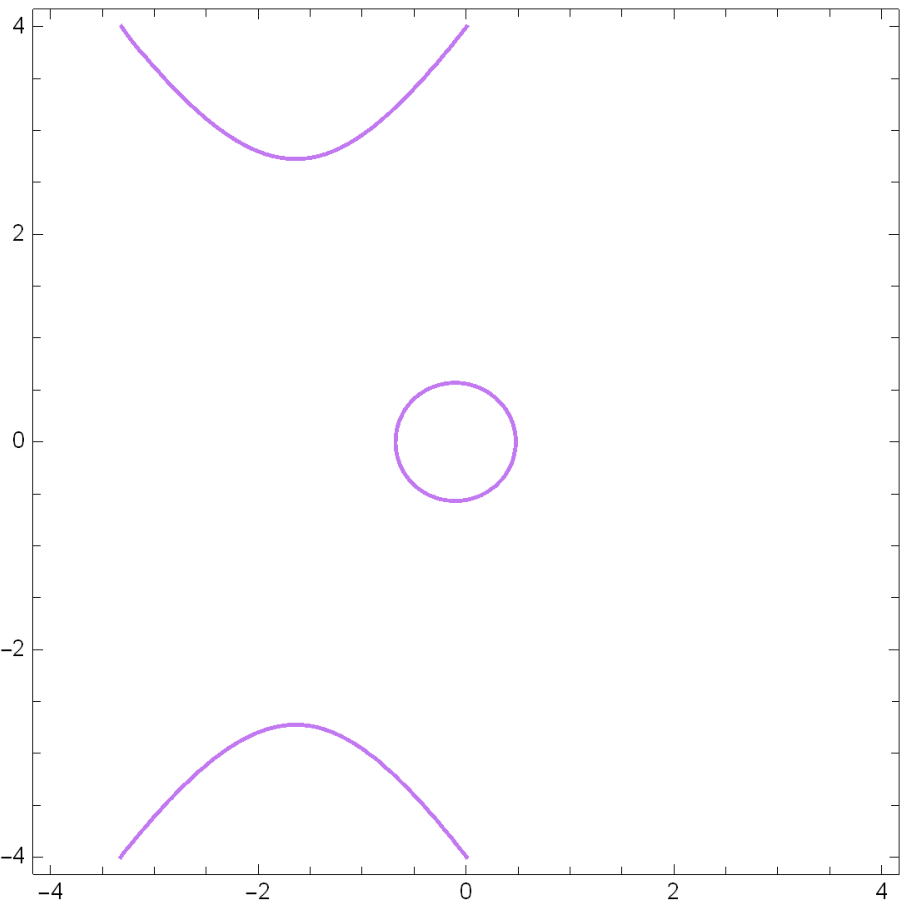}
  \caption{}\label{foc7}
   \end{minipage}\hfill
   \begin{minipage}{0.3\textwidth}
     \centering
     \includegraphics[height=4cm, width=4cm]{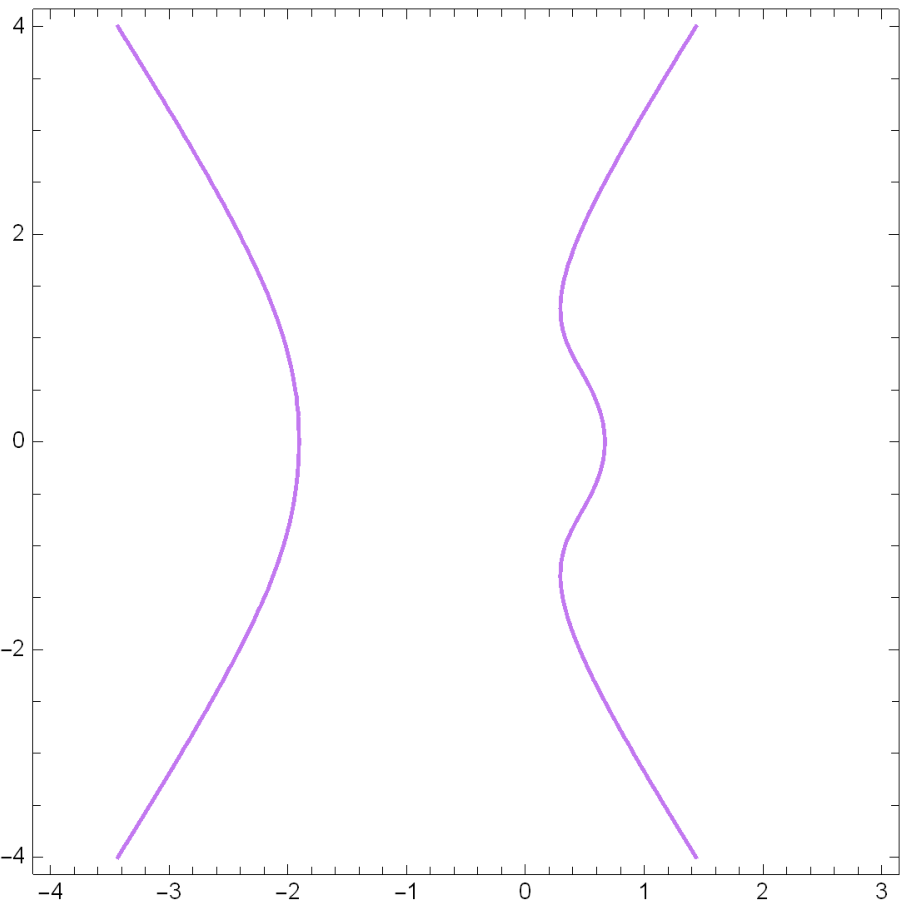}
 \caption{}\label{foc8}
   \end{minipage}\hfill
   \begin{minipage}{0.3\textwidth}
     \centering
     \includegraphics[height=4cm, width=4cm]{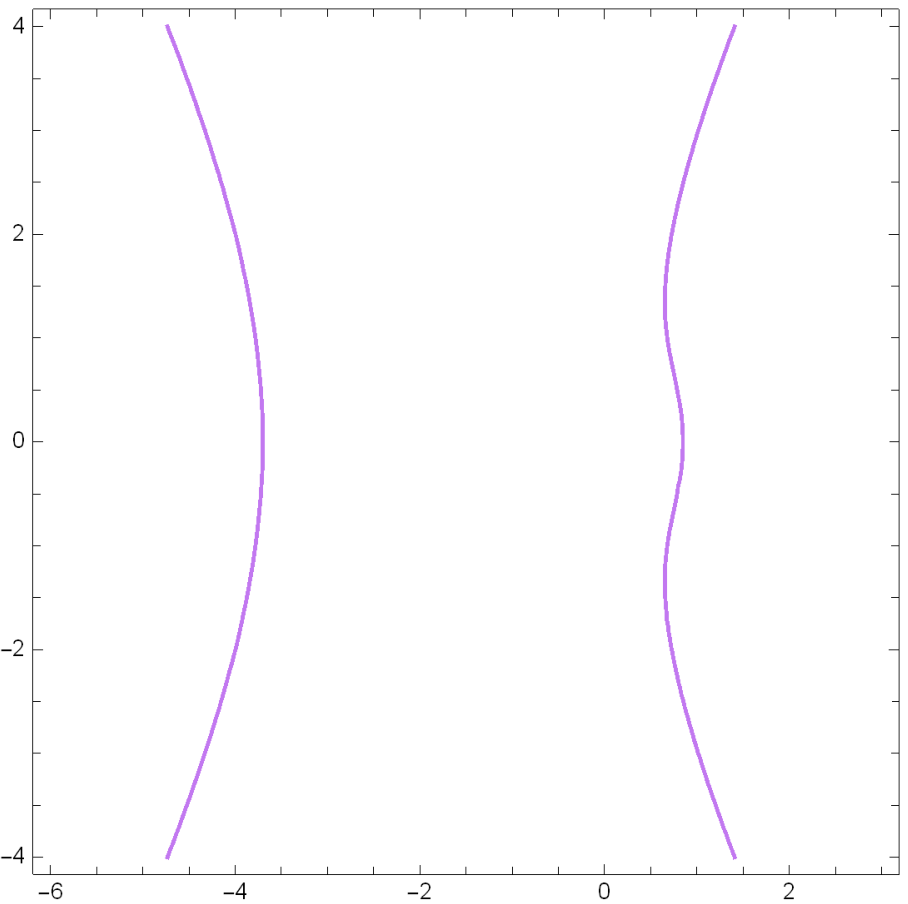}
\caption{}\label{foc9a}
   \end{minipage}
 \end{figure}
 We conclude that $\mathcal{C}$ contains  a Jordan curve with $0$ in the interior if and only if $(\gamma, v)\in \Omega_1$ corresponding to $(\alpha,1,\gamma) \in \Omega$ under the above change of variables. For an  arbitrary $\beta=m$, we define $L$ by  $v=-\frac{1}{24}\beta^2$ which is equivalent to $12\alpha\gamma-3m=0$ which proves the Lemma. 
 \end{proof}
 
 \begin{remark}
As a set, $\mathcal{C}$ is $b^{-1}(\mathbb{R})$ with the real axis removed. Therefore 
a Jordan curve found in  $\mathcal{C}$ is precisely the same curve in $b^{-1}(\mathbb{R})$.
\end{remark}
\begin{figure}[!htb]
   \begin{minipage}{0.33\textwidth}
     \centering
     \includegraphics[height=4cm, width=4cm]{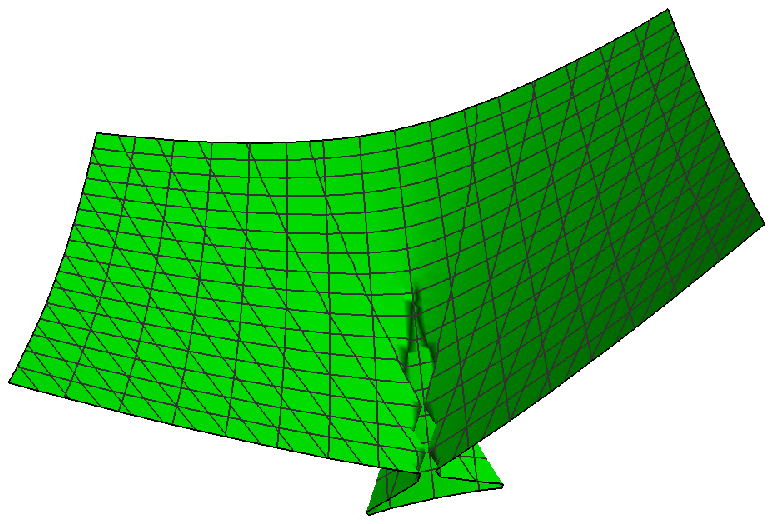}
 \caption{ }\label{foc1}
   \end{minipage}\hfill
   \begin{minipage}{0.33\textwidth}
     \centering
     \includegraphics[height=4cm, width=4cm]{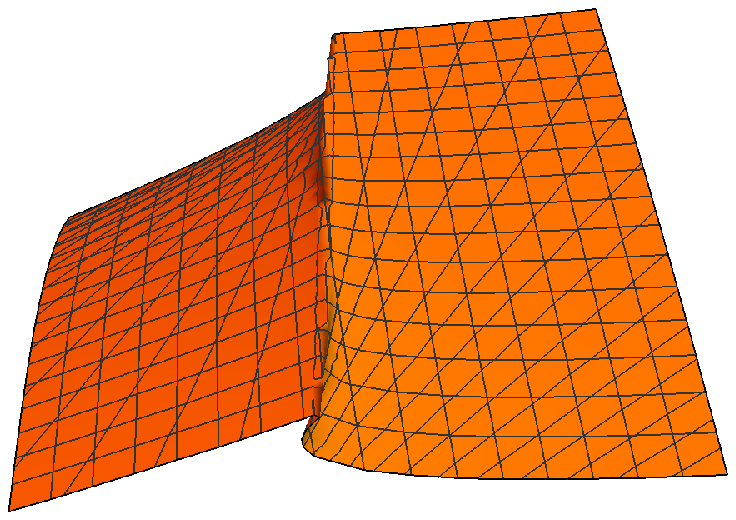}
\caption{}\label{foc2}
   \end{minipage}\hfill
   \begin{minipage}{0.3\textwidth}
     \centering
     \includegraphics[height=4cm, width=5cm]{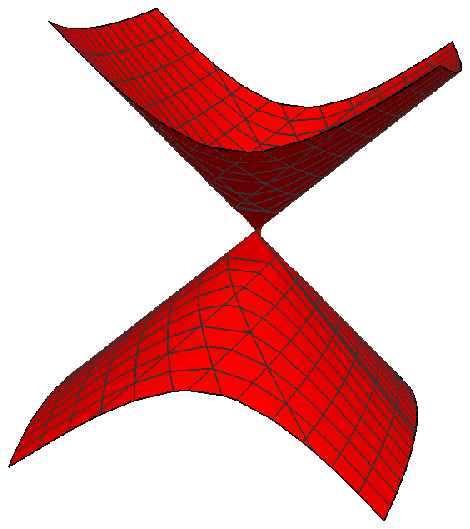}
\caption{ }\label{foc3x}
   \end{minipage}
 \end{figure}

Figure \ref{foc1} represents the zero set of the discriminant of $G(z)$ i.e., the points
$(\alpha,\beta, \gamma)\in \mathbb{R}^3$ satisfying $3\alpha^4 \gamma-\alpha^3 \beta^2-72 \alpha^2 \gamma^2+108 \alpha \beta^2 \gamma-27 \beta^4+432 \gamma^3=0.$
It is called a swallowtail in the literature on singularities and  is a surface in the 3-dimensional space of real polynomials of the form $z^4+ az^2+bz+c$, for $(a,b,c)$ corresponding to polynomials with multiple real roots. 

Figure \ref{foc2} represents the zero set of the discriminant of $H(z)$ i.e., the points $(\alpha,\beta, \gamma)\in \mathbb{R}^3$ satisfying $16 \gamma^2 \left(\alpha^2+4 \gamma\right)-8 \alpha \beta ^2 \gamma+\beta ^4=0$. This surface is the vanishing set of the discriminant  of $P_3(z)$  defined by the recurrence in Problem \ref{prob:main} whereby  $\alpha\mapsto \sqrt[3]{2} \gamma$ and $\beta \mapsto  \frac{4\alpha \gamma -\beta^2}{3\sqrt{3}}$. The surface intersects  the plane $\beta=1$ to form a cuspidal cubic (also called a semicubical parabola) whose  regular points correspond to polynomials with a double zero, and the vertex to the cubic polynomial $x^3$ which has a tripple zero.   

Figure \ref{foc3x}  represents the zero  set of the discriminant of $W(z)$ i.e., the points $(\alpha,\beta, \gamma)\in \mathbb{R}^3$ satisfying $12\alpha\gamma -3\beta^2=0$. It is a quadric surface and is the discriminant set of the quadratic polynomial
$W(z)=\alpha z^2+\beta z+3\gamma$.
\medskip

We now settle Theorem $2(b)$.

\medskip

\textit{Proof of Theorem 2 (b)}: Suppose that 
\begin{eqnarray}\label{cccc}
 \gamma <0,~~~\bigtriangleup(G)<0,~~\bigtriangleup(W)\leq 0 ~~\mbox{and}~~\bigtriangleup(H)> 0. \end{eqnarray}
By Lemma \ref{SH}, $b^{-1}(\mathbb{R})$ contains a smooth Jordan curve with $0$ in its interior. Hence Theorem $1(ii)$ is satisfied. By the equivalence $(ii)\Rightarrow (iii)$ of 
Theorem $1$, we deduce  that all the zeros of $P_n(z)$ are real for all positive integer $n$. 

Suppose instead of $ \bigtriangleup(H)>0$ we have $ \bigtriangleup(H)=0$ in \eqref{cccc}. Then 
 we obtain a limiting case where $b^{-1}(\mathbb{R})$ contains a non-smooth curve enclosing $0$. Since for each positive integer $n$, all the zeros of $P_n(z)$ are real in \eqref{cccc}, it follows that the zeros will remain  real in the limit when $ \bigtriangleup(H)=0$. (Notice that at least one of these real-rooted polynomials must have a multiple real zero). This completes the proof.

\section{Examples}\label{BIII}\label{sect6}
In this section, we present several concrete examples using numerical experiments.
\begin{example}
\end{example}\noindent  
We illustrate the case  Theorem 2 ($a$) where $G(z)$ is hyperbolic, and plot the zeros of some polynomials $P_n(z)$ in the sequence $\{P_n(z) \}_{n=1}^{n=\infty}$.
For example, we choose $(\alpha, \beta, \gamma)=(-\frac{27}{4}, -\frac{7}{8}, \frac{5}{2}), (-11, 9, -\frac{8}{3}),  (-\frac{26}{10}, 1, -\frac{1}{10}), (-2, \frac{2}{10}, \frac{2}{10}), (-2, 0, \frac{1}{3}),$ and $(-6,-4,-1)$ corresponding to $G(z)$ having :- distinct roots, (see Fig. \ref{fig32}); one double real root and two simple zeros  in which the multiple zero has a larger absolute value than the simple zero of the same sign, (see Fig. \ref{fig42}); one double real root and two simple zeros  in which the multiple zero is smaller in absolute value than the simple zero of the same sign (see Fig. \ref{fig52}); one double real root and two simple zeros  in which the two simple zeros have the same sign, (see Fig. \ref{fig62}); two pairs of equal real zeros, (see Fig. \ref{fig72}) and  all the zeros are equal in which three are equal, (see Fig. \ref{fig82}) respectively. We note that same parameters were used to obtain  Figures $\ref{fig3}- \ref{fig8}$ in that order. In Figures $\ref{fig32}- \ref{fig82}$, we show a portion of the curve $b^{-1}(\mathbb{R})$ and the zeros of $P_{150}(z)$.
\begin{figure}[!htb]
   \begin{minipage}{0.3\textwidth}
     \centering
     \includegraphics[height=4cm, width=4cm]{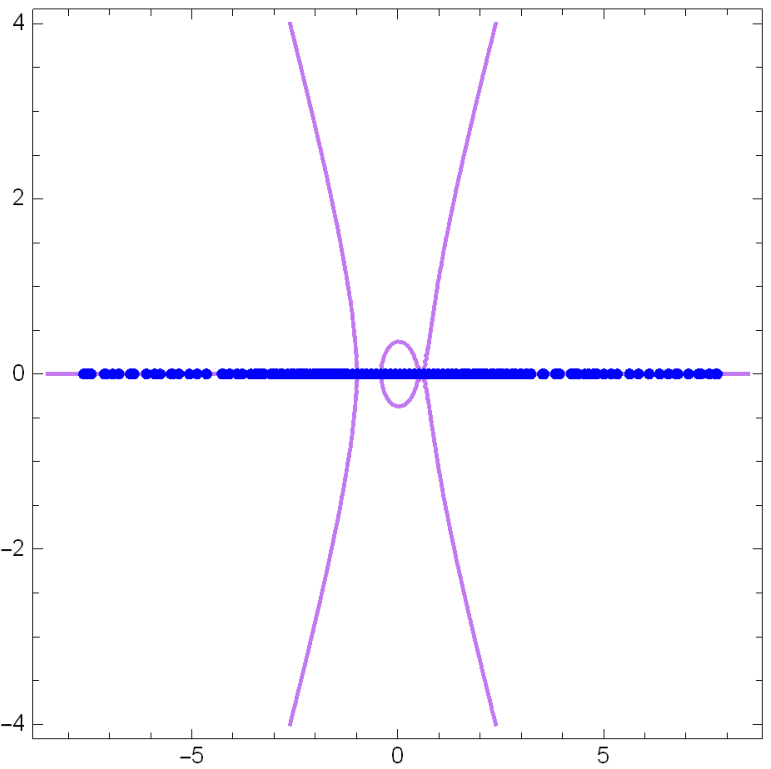}
     \caption{}\label{fig32}
   \end{minipage}\hfill
   \begin{minipage}{0.3\textwidth}
    \centering
     \includegraphics[height=4cm, width=4cm]{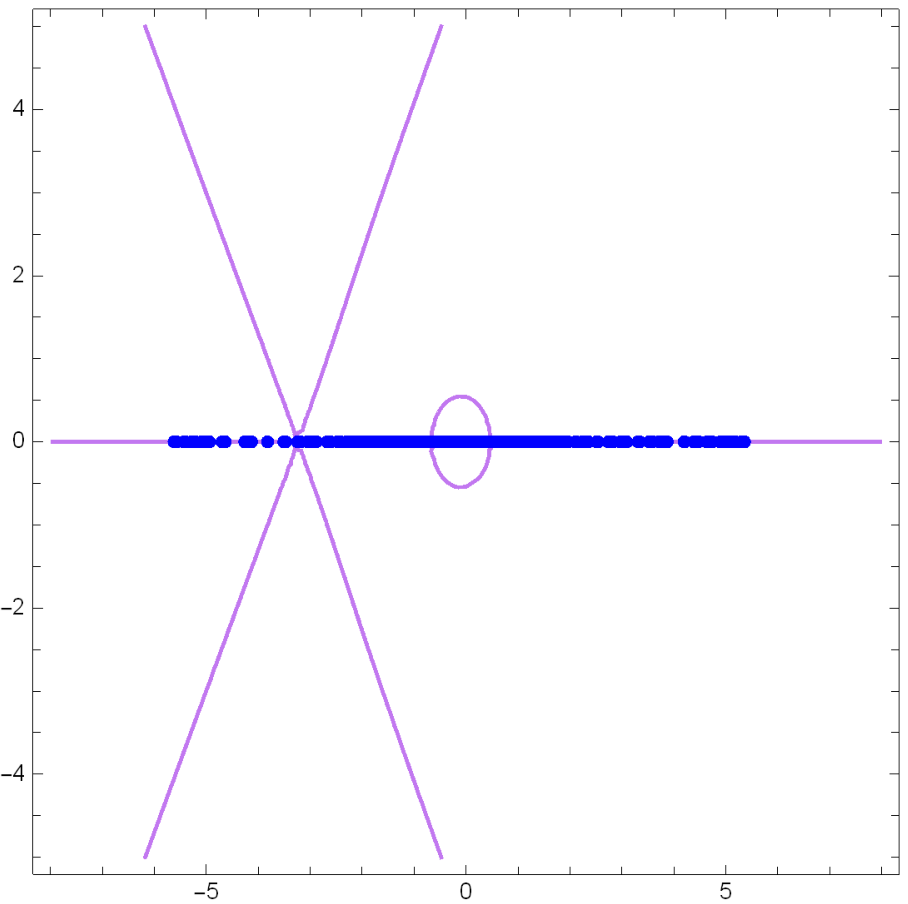}
    \caption{}\label{fig42}
   \end{minipage}\hfill
   \begin{minipage}{0.3\textwidth} 
     \centering
     \includegraphics[height=4cm, width=4cm]{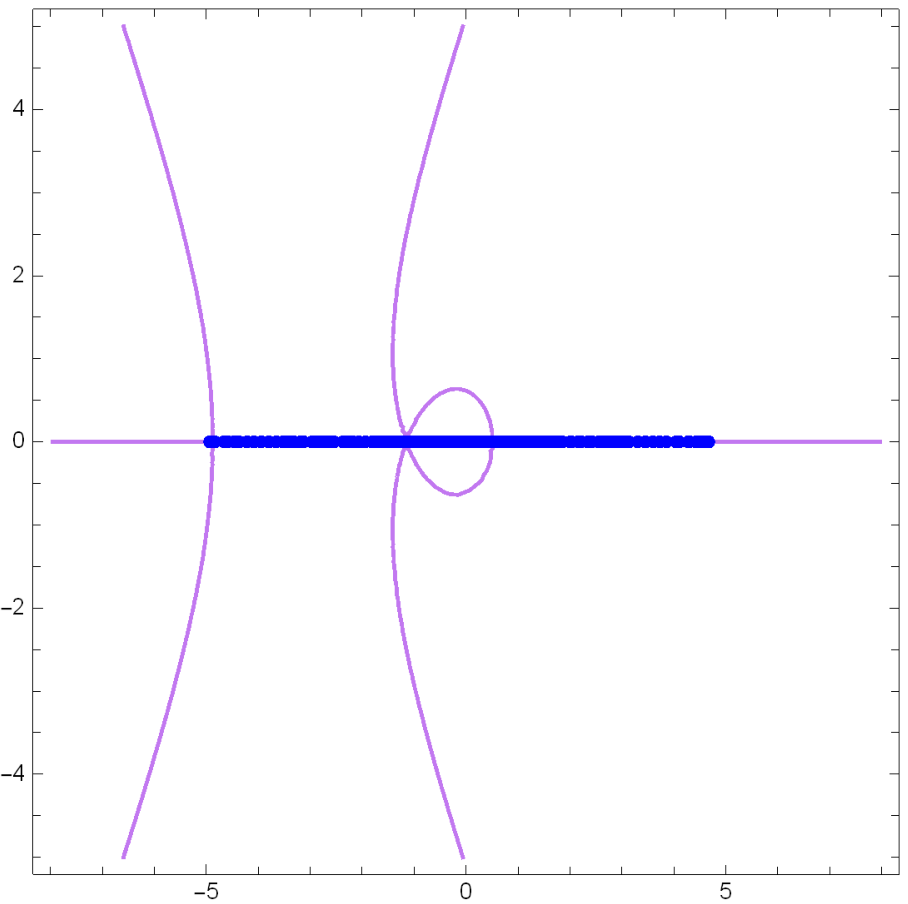}
     \caption{}\label{fig52}
   \end{minipage}
 \end{figure}
\medskip
  \begin{figure}[!htb]
   \begin{minipage}{0.3\textwidth}
     \centering
     \includegraphics[height=4cm, width=4cm]{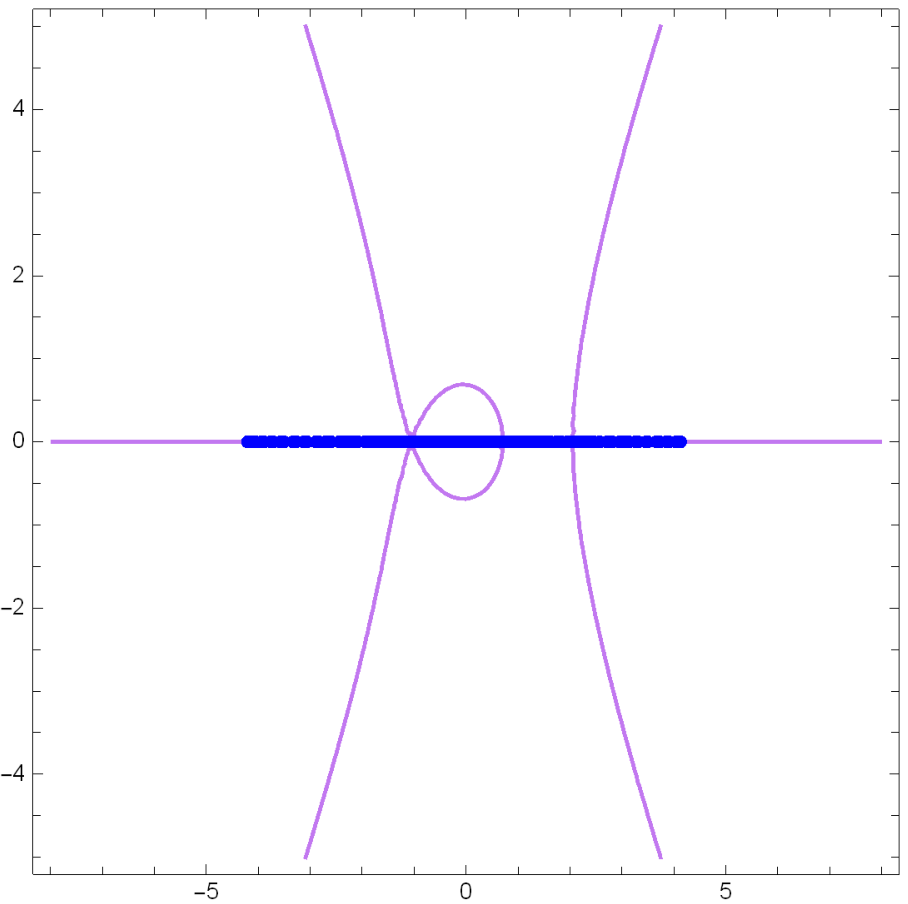}
     \caption{}\label{fig62}
   \end{minipage}\hfill
   \begin{minipage}{0.3\textwidth}
    \centering
     \includegraphics[height=4cm, width=4cm]{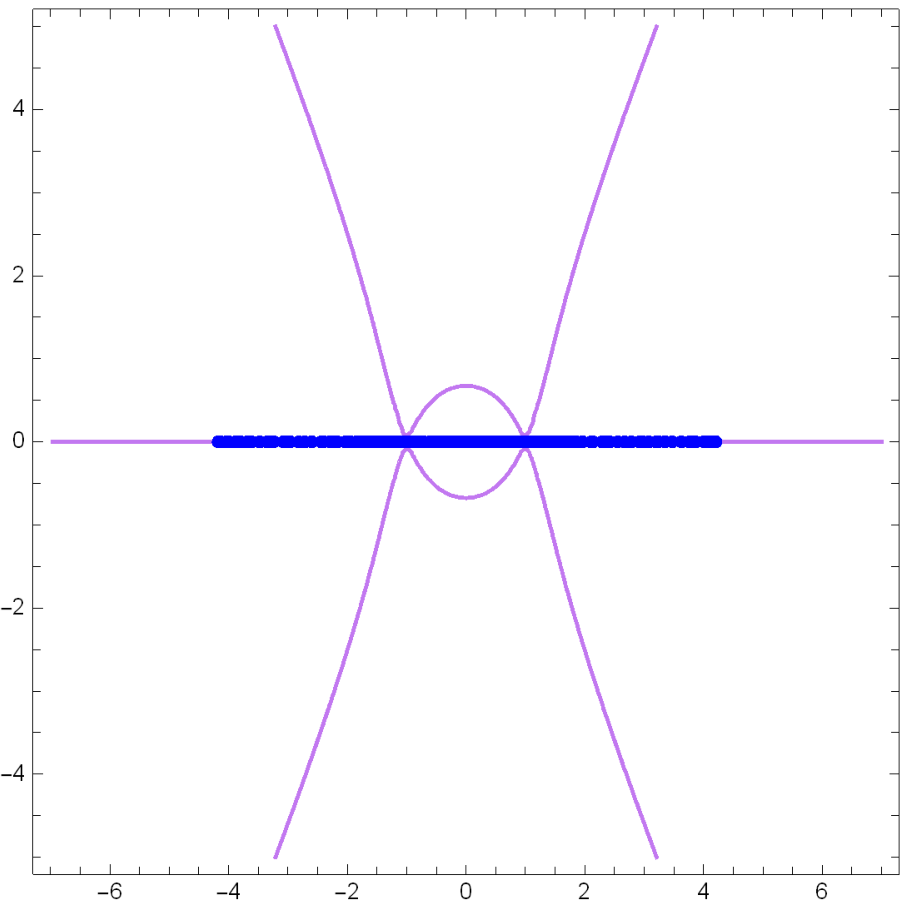}
    \caption{}\label{fig72}
   \end{minipage}\hfill
   \begin{minipage}{0.3\textwidth}
     \centering
     \includegraphics[height=4cm, width=4cm]{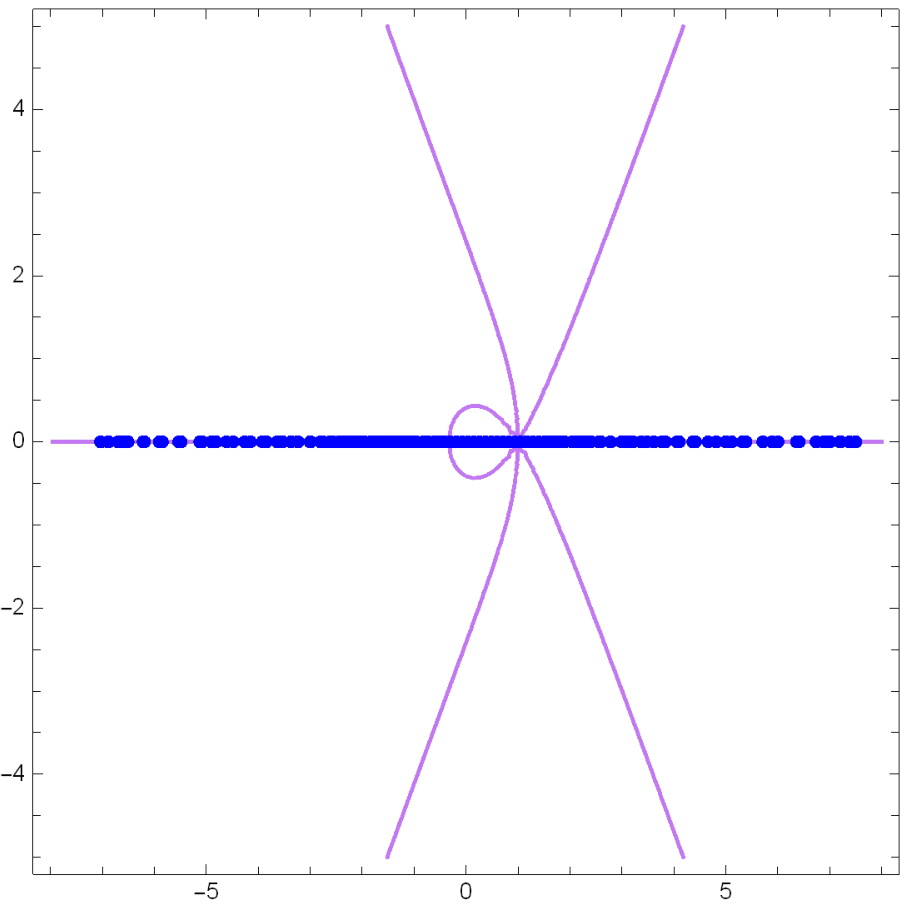}
     \caption{}\label{fig82}
   \end{minipage}
 \end{figure}
\begin{example} 
\end{example} \noindent 
We illustrate the case Theorem 2 ($b$). We use  $(\alpha, \beta, \gamma)= (-\frac{13}{4}, 1, -\frac{1}{5}), (-\frac{9}{20}, 1, -\frac{1}{3})$ and $(\frac{3}{4}, 1, -\frac{2}{10})$ respectively. These are the same parameters used to obtain Figures $\ref{foc7}- \ref{foc9a}$ in that order. In Figures $\ref{cap1}- \ref{cap3}$, we show  the portion of the curve $b^{-1}(\mathbb{R})$ and the zeros of $P_{70}(z)$.
\begin{figure}[!htb]
   \begin{minipage}{0.3\textwidth}
     \centering
     \includegraphics[height=4cm, width=4cm]{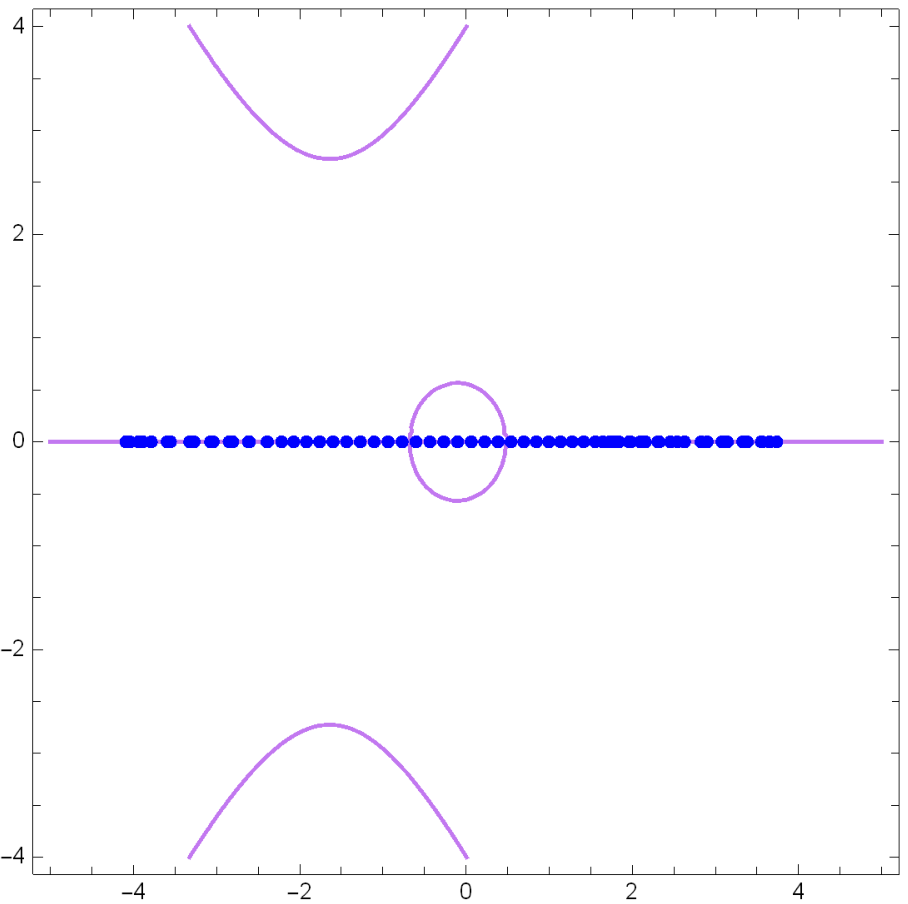}
  \caption{}\label{cap1}
   \end{minipage}\hfill
   \begin{minipage}{0.3\textwidth}
     \centering
     \includegraphics[height=4cm, width=4cm]{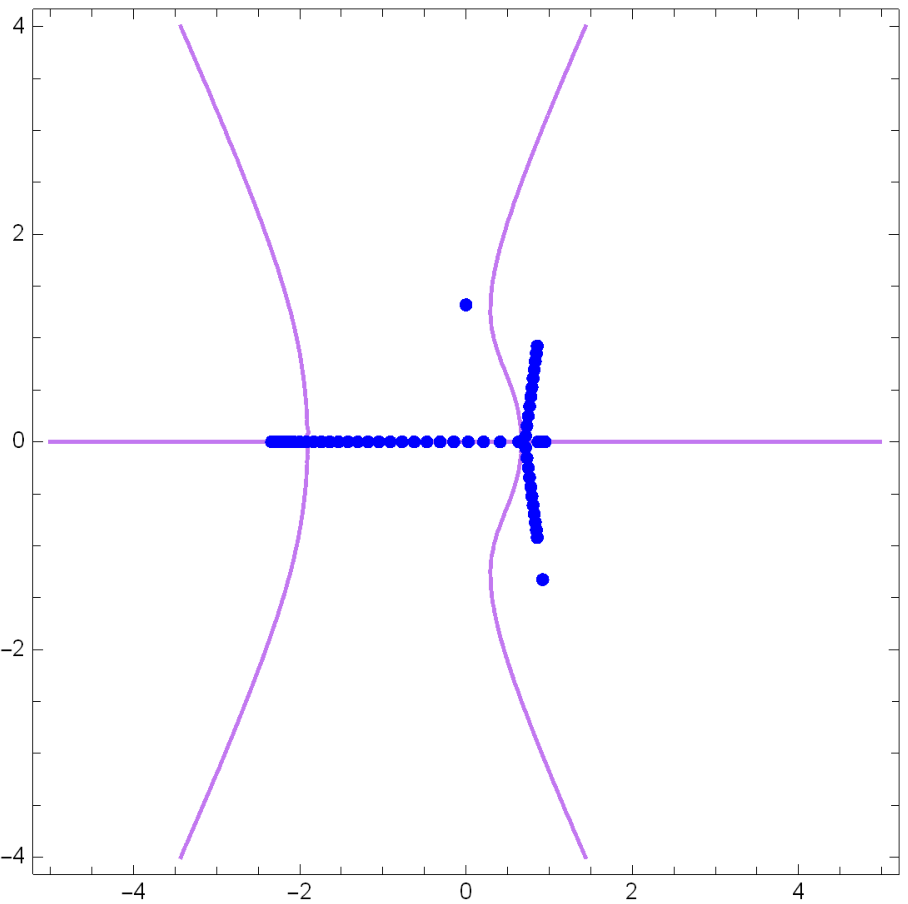}
 \caption{}\label{cap2}
   \end{minipage}\hfill
   \begin{minipage}{0.3\textwidth}
     \centering
     \includegraphics[height=4cm, width=4cm]{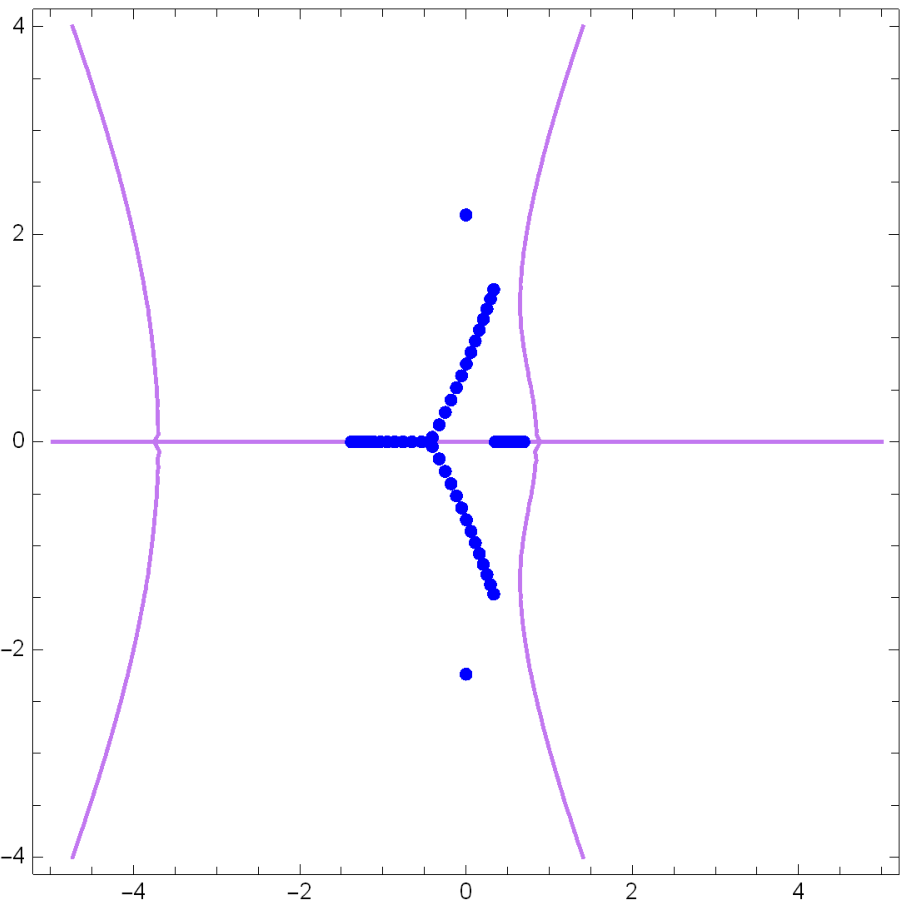}
\caption{}\label{cap3}
   \end{minipage}
 \end{figure}
 
 \begin{example}
\end{example}\noindent  
Fix a pair $(\alpha, \beta)=(-5,4)$ and vary the parameter $\gamma \in \mathbb{R}$ thus getting the symbol
$$ b(z)=-\frac{1}{z} -5z- 4z^2 + \gamma z^3.$$
Consider  $\gamma=-13,-2, -0.7$. Then for any of these choices of $\gamma$, the polynomial $G(z)$ has two real roots (of opposite signs) and a pair of complex conjugate roots. 
In Figures $\ref{Ex1}- \ref{Ex3}$, we show the portion of the curve $b^{-1}(\mathbb{R})$ and the zeros of $P_{30}(z)$ for $\gamma=-13,-2, -0.7$ respectively.
\begin{figure}[!htb]
   \begin{minipage}{0.3\textwidth}
     \centering
     \includegraphics[height=4cm, width=4cm]{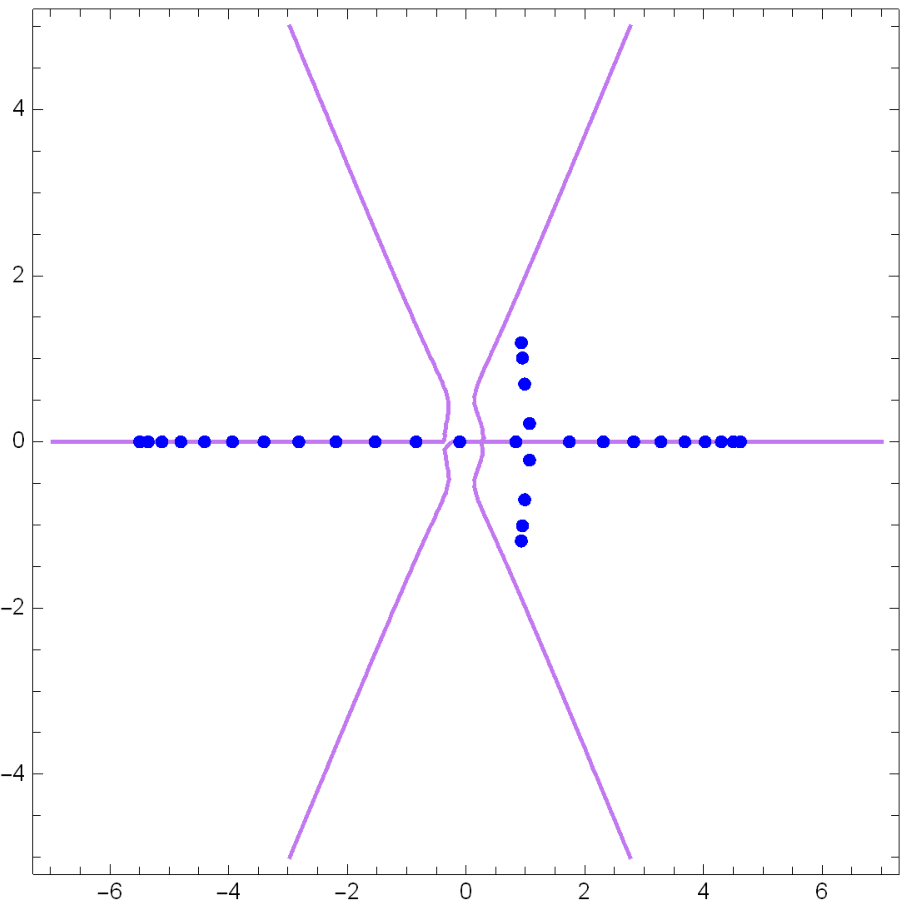}
  \caption{}\label{Ex1}
   \end{minipage}\hfill
   \begin{minipage}{0.3\textwidth}
     \centering
     \includegraphics[height=4cm, width=4cm]{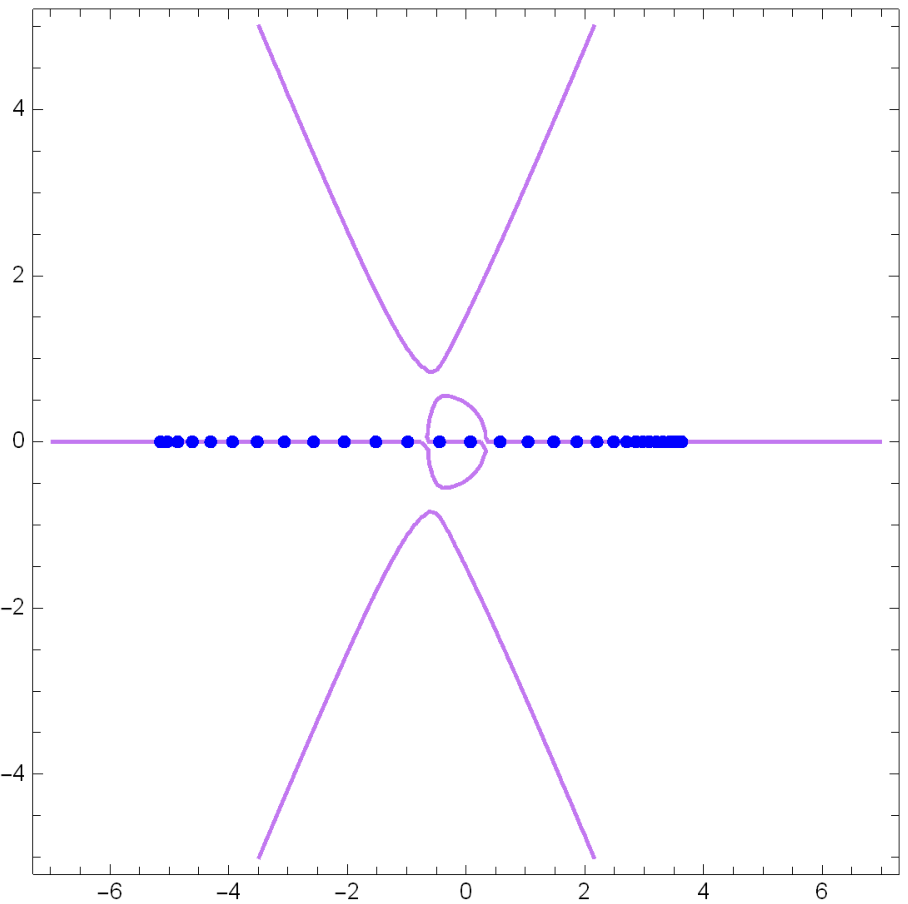}
 \caption{}\label{Ex2}
   \end{minipage}\hfill
   \begin{minipage}{0.3\textwidth}
     \centering
     \includegraphics[height=4cm, width=4cm]{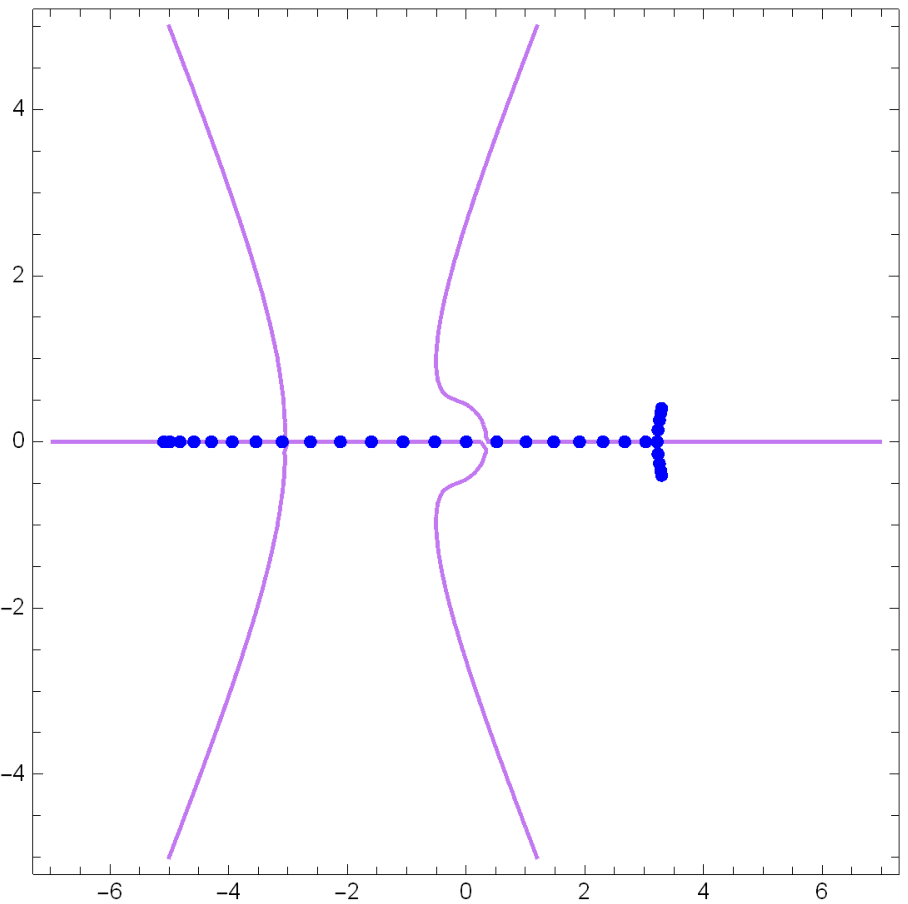}
\caption{}\label{Ex3}
   \end{minipage}
 \end{figure}

\begin{remark}
From our plots, we observe that whenever  $b^{-1}(\mathbb{R})$ contains a smooth Jordan curve with $0$ in its interior, $P_{n}(z)$ is real-rooted. This holds for any arbitrary positive integer $n$. 
 \end{remark}

 \section{Final Remarks}
 
 \noindent
{\bf 1.} A natural continuation of this project is to extend the results in Theorem 2 to recurrences of length greater than $5$.

\noindent
{\bf 2.} Each of the inequalities appearing in Theorem 2 can be interpreted in terms of the positivity/negativity of the discriminant of a certain polynomial with coefficients being explicit functions of the parameters $\alpha, \beta$ and $\gamma$. It would be interesting to have an interpretation of the latter polynomials. One can guess that even for linear recurrences of length greater than $5$, the conditions for having hyperbolic polynomials $P_n(z)$ for all $n \in  \mathbb{N_+}$ can be described as a set of discriminantal inequalities. 
  \medskip
  
  {\bf Acknowledgements.}  I am sincerely grateful to my advisor Professor Boris Shapiro who introduced me to this interesting problem and for many fruitful discussions surrounding it. I also thank Dr. Alex Samuel Bamunoba for the discussions and guidance. I acknowledge and appreciate the financial support from Sida Phase-IV bilateral program with Makerere University 2015-2020 under project 316 `Capacity building in mathematics and its applications'.

  \end{document}